\documentclass[11pt, reqno]{amsart}
\usepackage{amssymb,amsmath,amsopn,textcomp, thmtools, enumerate}

\usepackage[backgroundcolor=white, bordercolor=blue, 
linecolor=blue]{todonotes}
\usepackage[normalem]{ulem}

\usepackage[a4paper,margin=1in,footskip=0.25in]{geometry}
\usepackage[ansinew]{inputenc}
\usepackage[english]{babel}
\usepackage{mathrsfs}
\usepackage{setspace}
\usepackage{stmaryrd}
\usepackage[titletoc]{appendix}
\usepackage[colorlinks=true, linkcolor=blue, citecolor=blue]{hyperref}

\usepackage{dsfont}
\usepackage{bbm}

\usepackage{graphicx}
\usepackage{color}
\usepackage{caption,rotating}

\newtheorem{theorem}{Theorem}[section]

\newtheorem{proposition}[theorem]{Proposition}
\newtheorem{lemma}[theorem]{Lemma}
\newtheorem*{theorem*}{Theorem}

\theoremstyle{definition}

\theoremstyle{remark}
\newtheorem{rem}[theorem]{Remark}

\allowdisplaybreaks

\renewenvironment{proof}[1][Proof]{\noindent\textit{#1} }{\hfill 
	\rule{0.5em}{0.5em}}

\addto\extrasenglish{}
\addto\extrasenglish{} 

\numberwithin{equation}{section}

\usepackage{color}

\newcommand{\cA}{\mathcal{A}}

\newcommand{\cD}{\mathcal{D}}
\newcommand{\cL}{\mathcal{L}}

\newcommand{\cE}{\mathcal{E}}
\newcommand{\cF}{\mathcal{F}}

\newcommand{\cH}{\mathcal{H}}

\newcommand{\bN}{\mathbb{N}}
\newcommand{\bH}{\mathbb{H}}
\newcommand{\bR}{\mathbb{R}}
\newcommand{\bP}{\mathbb{P}}

\newcommand{\Rd}{{\mathbb{R}^d}}

\newcommand{\E}{\mathbb E}
\newcommand{\Pp}{\mathbb P}
\newcommand{\1}{\mathbbm 1}

\newcommand{\eps}{\varepsilon}
\newcommand{\I}{{\rm I}}
\newcommand{\II}{{\rm II}} 
\newcommand{\III}{{\rm III}}

\def\nn{\nonumber}

\newcommand\FF{\mathcal {F}}

\newcommand{\lC}{{\underline{c}}}
\newcommand{\uC}{{\overline{C}}}
\newcommand{\la}{{\underline{\alpha}}}
\newcommand{\ua}{{\overline{\alpha}}}

\newcommand{\wh}[1]{{\widehat{#1}}}
\newcommand{\wt}[1]{{\widetilde{#1}}}
\def\ol{\overline}

\makeatletter
\@namedef{subjclassname@2020}{%
	\textup{2020} Mathematics Subject Classification}
\makeatother

\author{Kyung-Youn Kim}
\email{kykim@nchu.edu.tw}
\address{Department of Applied Mathematics, National Chung Hsing University, Taichung 402, Taiwan}

\author{Lidan Wang}
\email{lidanw.math@gmail.com}
\address{School of Statistics and Data Science, Nankai University, Tianjin, China}

\title[Dirichlet heat kernel bounds for systems of nonlocal equations]
{Dirichlet Heat kernel estimates for a large class of anisotropic Markov processes}
\thanks{KY's Research  partially supported by NSTC of Taiwan (Grant No. NSTC-110-2115-M-005-009-MY3) and	
	LW's Research partially supported by NNSF of China (Grant No. 11801283).} 


\begin{document}

	\begin{abstract}
	Let $Z=(Z^{1}, \ldots, Z^{d})$ be the d-dimensional L\'evy {process} where {$Z^i$'s} are  independent 1-dimensional L\'evy {processes} with  identical jumping kernel $ \nu^1(r) =r^{-1}\phi(r)^{-1}$. Here $\phi$ is {an} increasing function with weakly scaling condition of order $\underline \alpha, \overline \alpha\in (0, 2)$. 
We consider a symmetric function $J(x,y)$ comparable to 
\begin{align*}
	\begin{cases}
		\nu^1(|x^i- y^i|)\qquad&\text{ if  $x^i \ne y^i$ for some $i$ and $x^j = y^j$ for all $j \ne i$}\\
		0\qquad&\text{ if $x^i \ne y^i$ for more than one index $i$}.
	\end{cases}
\end{align*}
Corresponding to the jumping kernel  $J$, there exists an anisotropic Markov process $X$, see \cite{KW22}. In this article, we establish sharp two-sided Dirichlet heat kernel  estimates for $X$ in $C^{1,1}$ open set, under certain regularity conditions.
As an application of the main results, we derive the Green function estimates.
\end{abstract}
\maketitle
	
\setstretch{1.1}
\section{Introduction}

It is a well-known fact that the fundamental solution $p(t, x, y)$ of the heat equation with  a second order elliptic differential operator on $\bR^d$ is the transition density of the diffusion process related to the operator as an infinitesimal generator.
The relation is also true for a large class of Markov processes corresponding to non-local operators with discontinuous sample paths. Correspondingly, there are lots of studies on the sharp two-sided heat kernel estimates for  isotropic Markov processes (see, \cite{ChKu03, ChKu08, CKK11} and other references therein).

\medskip
For any open set $D\subset \bR^d$, one can define a killed processes upon leaving $D$ and there exists a transition density $p_D(t, x, y)$ if the distribution is absolutely continuous.
The function $p_D(t, x, y)$ is also called the Dirichlet heat kernel which describes an operator with zero exterior conditions. The  Green function and solutions to Cauchy and Poisson problems with Dirichlet conditions can be expressed according to the Dirichlet heat kernel as well.
However, it is difficult to study the Dirichlet heat kernel estimates for $p_D(t, x, y)$, especially when the process is close to the boundary.
Therefore, the Dirichlet heat kernel estimates were obtained relatively recently for the Laplacian  operator (corresponding to the Brownian motion) in \cite{Da87, H92} for the upper bound, and in \cite{Zh02} for the lower bound.
In \cite{CKS12},  the authors estimated the Dirichlet heat kernel for the fractional Laplacian operator (corresponding to the isotropic stable L\'evy process) in $C^{1,1}$ open set over finite time intervals, and in bounded $C^{1,1}$ open set for a large time.
The elegant techniques developed in \cite{CKS12} provide the framework on the Dirichlet heat kernel estimates for non-local operators. With successful application and extension of the framework, there are lots of results on  discontinuous Markov processes: isotropic L\'evy processes in \cite{BGR14_D, CKS14, CKK21} ,  symmetric L\'evy processes with Gaussian components in \cite{CKS16}  and symmetric Markov process (beyond the L\'evy process) in \cite{KK, GKK20}.
Also, such non-local operators can be used to describe models in diverse phenomena in the real world, see, \cite{CT04, CV10, LS10, HL12}.

\medskip
For the last two decades, there have been  increasing interests in the anisotropic Markov processes. 
Consider the system of stochastic differential equations
$$dZ_t^i=\sum_{j=1}^d A_{ij}(Z_{t-})dL_t^j, \qquad i=1,\ldots, d.$$
where $L^i$ are identically distributed one-dimensional $\alpha$-stable process  with $\alpha\in(0,2)$.
In \cite{BaCh06}, under mild conditions on the matrix $A=(A_{ij})$, the authors observed the unique weak solution to this system which forms a strong Markov process.
In \cite{KuRy18, KuRy20}, they considered more specific conditions for $A$ and $\alpha$ to study the corresponding operators, and in \cite{Cha19, Cha20, KRS21}, they considered $L^i$ which are not necessarily same processes.  It is worth to mention that there are fundamental differences between isotropic and anisotropic processes. For example, it is shown in \cite{BaCh10} that Harnack inequality does not hold for cylindrical $\alpha$-stable processes. Intuitively, the significant difference lies in the fact that isotropic Markov processes can move (or jump) in any direction uniformly, while anisotropic Markov processes can only jump along one of the coordinate directions.

\medskip
As for heat kernel estimates, they were firstly obtained (not in a sharp form) in \cite{Xu13} for an anisotropic Markov process whose jumping measure is comparable to that of cylindrical $\alpha$-stable processes. With delicate analytic methods, sharp two-sided heat kernel estimates were established  in 
\cite{KKK22}. The results have been furtherly extended to more general anisotropic Markov processes in \cite{KW22}. Estimates of Dirichlet heat kernels for anisotropic Markov processes is a completely new challenge. Very recently, in \cite{CHZ}, two-sided Dirichlet heat kernel estimates are studied when $L^i$ are $\alpha$-stable processes.
In this article, we consider more general anisotropic Markov processes introduced in \cite{KW22} and establish the two-sided Dirichlet heat kernel estimates in $C^{1,1}$ domains.

\medskip
For any $0<\la\le \ua<2$, let $\phi:[0, \infty)\rightarrow[0, \infty)$ be an increasing function with the following condition:
there exist positive constants $\lC\le 1$ and $\uC\ge 1$ such that 
\begin{align*}
	{\bf(WS):}
	\qquad\qquad	
	\lC\left(\frac{R}{r}\right)^{\la}\le \frac{\phi(R)}{\phi(r)}\le \uC \left(\frac{R}{r}\right)^{\ua} \qquad \mbox{ for }\,\, 0<r\, \le  R.
\end{align*}
Using this $\phi$,  define
\begin{align}\label{d:nu1}
	{\nu^1}(r):={(r \phi(r))}^{-1}\qquad\mbox{ for}\,\, r>0.
\end{align}
Then ({\bf WS}) implies 
$$
\int_\bR(1\wedge |r|^2){\nu^1}(|r|)dr\le c \left(\int_0^{1}r^{-\ua+1} d r+ \int_{1}^{\infty} r^{-\la-1} \right)d r<\infty
$$
for some $c>0$,
so ${\nu^1}(dr):={\nu^1}(|r|)dr$ is a L\'{e}vy measure.
Consider {an anisotropic} L\'evy process $Z$ in $\bR^d$ defined by $Z =(Z^{1}, \ldots, Z^{d})$, where each coordinate process $Z^{i}$ is an independent one-dimensional symmetric L\'evy process with L\'evy measure ${\nu^1}(dr)$.
Then the corresponding L\'evy measure $\nu$ of $Z$ is represented as
\begin{align*}
	\nu(d h) = \sum\limits_{i=1}^d {\nu^1}(|h^i|) d h^i \prod\limits_{j \ne i} \delta_{\{0\}}(d h^j).
\end{align*}
Intuitively, $\nu$ only measures  the sets containing the line which is parallel to one of the coordinate axes.
The corresponding Dirichlet form $(\mathcal{E}^\phi,  \mathcal F^\phi)$ on $L^2(\bR^d)$ is given by 
\begin{align*}
	\mathcal{E}^\phi (u,v) &=\int_{\Rd}\Big(\sum_{i=1}^d\int_\bR\big(u(x+e_i \tau) - u(x)\big)\big(v(x+e_i \tau) -  v(x)\big) J^\phi (x, x+e_i \tau) d \tau \Big)d x,\notag \\
	\mathcal F^\phi &=\{u\in L^2(\Rd)| \mathcal{E}^\phi (u,u)<\infty\},
\end{align*}
where $e_i$ is the unit vector in the positive $x^i$-direction and $J^\phi(x,y)$ is the jumping kernel defined as follows:
$$J^\phi (x, y):=
\begin{cases}
|x^i-y^i|^{-1}\phi(|x^i-y^i|)^{-1} \qquad&
\text {if $x^i \ne y^i$ for some $i$; $x^j = y^j$ for all $j \ne i$,}\\
0\,&\text{otherwise}.
\end{cases}
$$
We also consider a symmetric measurable function
$\kappa:\Rd\times\Rd\to (0,\infty)$ satisfying that
there exists  $\kappa_0>1$ such that 
\begin{align}\label{a:kappa}
	\kappa_0^{-1}\leq \kappa(x,y)\leq \kappa_0\qquad\text{ for }
	x,\,y\in\Rd.
\end{align}
Let $J:\Rd\times\Rd \setminus \operatorname{diag}\to (0,\infty)$ be a symmetric measurable function defined as
\begin{align}\label{a:J}
	J(x, y): =  \kappa(x,y)   J^\phi (x, y).
\end{align}	
By ({\bf WS}), there exists $c>0$ such that for any constant $a>0$ and any $i\in\{1,\dots,d\}$, we have
\begin{align}\label{e:large}
	\int_{\{|\tau|>a\}}\hspace{-.2in} J(x, x+e_i\tau)d\tau< \frac{c}{\phi(a)}\  \text{ for }x\in \Rd\quad  \text{ and }\quad 
	\int_{\{|\tau|>a\}}\hspace{-.2in} \nu^1(|\tau|)d\tau<\frac{c}{\phi(a)}.
\end{align}
Consider the Dirichlet form $(\mathcal{E}, \FF)$ on $L^2(\Rd)$ as follows:
\begin{align*}
	\mathcal{E} (u,v)&
	:=\int_{\Rd}\Big(\sum_{i=1}^d\int_\bR\big(u(x+e_i \tau) - u(x)\big)\big(v(x+e_i \tau) -  v(x)\big) J (x,x+e_i \tau) d \tau \Big) d x \,,\nn\\
	\mathcal F&:=\{u\in L^2(\Rd)| \; \mathcal{E}(u,u)<\infty\}.
\end{align*}
Then we obtain in  \cite[Theorem 1.1]{KW22} that the existence of a conservative Feller process  $X=(X^1,\dots,X^d)$ associated with $(\cE,\cF)$.  Moreover,$X$ has a jointly continuous transition density function $p(t,x,y)$ on $\bR_+\times\bR^d\times\bR^d$, which enjoys the following estimates: there {exists a constant} $C_{1}> 1$ such that for any $t>0, x,y\in\bR^d$, 
\begin{align}\label{eq:hke}
	{C_{1}^{-1}}[\phi^{-1}(t)]^{-d}\prod_{i=1}^d&\left(1\wedge \frac{t\phi^{-1}(t)}{|x^i-y^i|\phi(|x^i-y^i|)}\right)\nn\\
	&
	\le p(t,x,y)\le {C_{1}}	[\phi^{-1}(t)]^{-d}\prod_{i=1}^d\left(1\wedge \frac{t\phi^{-1}(t)}{|x^i-y^i|\phi(|x^i-y^i|)}\right),
\end{align} 
where $a\wedge b:=\min \{a, b\}$. $X$ represents a large class of anisotropic Markov processes. More detailed examples can be found in \cite[Remark 1.2]{KW22}.

\medskip
For any open set $D\subset \bR^d$, we define the {\it first exit time} $\tau_D:= \inf\{t>0: X_t\notin D\}$ of $D$ for the process $X$,
and consider the subprocess  $X^D$ of $X$ killed upon leaving $D$ as follows:
$X^D_t(\omega)=X_t(\omega)$ if $t<\tau_D(\omega)$ and 
$X^D_t(\omega)=\partial$ if $t\ge \tau_D(\omega)$ where $\partial$ is a cemetery point.
In this article, we investigate that $X^D$ has a continuous transition density $p_D(t, x, y)$ with respect to the Lebesgue measure, then obtain the two-sided bounds for $p_D(t, x, y)$ in $C^{1, 1}$ open set $D\subset \bR^d$ for different ranges of time $t$.
\medskip

We call $D\subset \bR^d$  a $C^{1,1}$ open set if there exists a localization radius $R>0 $ and a constant $\Lambda>0$ such that for every $z\in\partial D$ there exists a $C^{1,1}$ function $\varphi=\varphi_z: \bR^{d-1}\to \bR$ satisfying 
\begin{align*}
	\varphi(0)=0,\ \nabla\varphi (0)=(0, \dots, 0), \ \| \nabla\varphi \|_\infty \leq \Lambda\ \text{ and }\ | \nabla \varphi(x)-\nabla \varphi(w)| \leq \Lambda |x-w|
\end{align*}and an orthonormal coordinate system $CS_z$ of  $z=(z^1, \cdots, z^{d-1}, z^d):=(\widetilde  z, \, z^d)$ with  an origin at $z$ such that $ D\cap B(z, R )= \{y=({\wt y}, y^d) \in B(0, R) \mbox{ in } CS_z: y^d > \varphi (\widetilde  y) \}$. 
The pair $( R, \Lambda)$ is called the $C^{1,1}$ characteristics of the open set $D$, and we may assume that $R<1$ and {$\Lambda>2$}. 
\medskip

Note that a $C^{1,1}$ open set $D$ with characteristics $(R, \Lambda)$ can be unbounded and disconnected, and the distance between two distinct components of $D$ is at least $R$.
$C^{1,1}$ open sets  in $\bR$ with a characteristic $R>0$ can be written as the union of disjoint intervals so that the infimum of the lengths of all these intervals is at least $R$.
It is well known that if $D$ is $C^{1,1}$ open set with the characteristics $(R, \Lambda)$, then $D$ satisfies the interior and exterior ball conditions with the characteristic $R_1\le R$, that is, there exist balls $B_1,B_2\subset \Rd$ with radius $R_1$ such that $B_1\subset D\subset B_2^c$ satisfying  $\delta_{B_1}(x)\le \delta_{D}(x)\le\delta_{B_2^c}(x)$ for any $x\in B_1$. Here we use the definition $\delta_A(x):=\text{dist}(x, \partial A)$ for any $A\subset \Rd$ and $x\in \Rd$. Throughout this article, without loss of generality, we always  assume that $R=R_1$.

\medskip
Throughout this article, we assume additional conditions for the regularities of $\kappa$ and $\phi$ appeared in \eqref{a:J}.
\begin{description}
	\item[(K$_\eta$)]
	There are $\kappa_1>0$ and $\ua/2<\eta\le1$ such that $|\kappa(x,x+h)-\kappa(x,x)|\leq \kappa_1|h|^{\eta}$ for every $x,\,h\in \Rd$, $|h|\leq 1$.
	\item[({\bf SD})]
	$\phi\in C^1(0,\infty) $ and $ r\to -(\nu^1)'(r)/r$  is decreasing.
\end{description}

\medskip
We now state our main results of this article.
Let 
\begin{align}\label{e:dax}
	\Psi(t,x):=\left(1\wedge\sqrt{\frac{\phi(\delta_D(x))}{t}}\right)\qquad \text{ for } t>0, x\in \bR^d.
\end{align}

\begin{theorem}\label{t:nmain}
	Suppose that $X$ is a symmetric pure jump Hunt process whose jumping intensity kernel $J$ satisfies the conditions \eqref{a:J}, $({\bf SD})$ and ${\bf (K_\eta)}$. 
	Suppose that  
	$D$ is a $C^{1,1}$ open set in $\bR^d$ with characteristics $(R, \Lambda)$.
	Then, for each $T>0 $, there exists $C_2=C_2(\phi, \kappa_0, \kappa_1,  \eta, R, \Lambda,  T, d)>0$ such that the transition density $p_D(t,x,y)$ of $X^D$ has the following upper estimate:
	\begin{equation}\label{eq:dhke_u}
		p_D(t, x, y)\le {C_2} \Psi(t,x) \Psi(t,y) \,p(t, x, y),\quad \text{ for any }(t,x,y)\in(0,T]\times D\times D.
	\end{equation}
\end{theorem}

\medskip
For the lower bound estimate, we need extra conditions:
for any $u\in \bR^d$, $a\in \bR$ and $1\le i\le d$, 
define a new point  
$$[u]_a^i:=(u^1, \ldots,u^{i-1}, a,{u^{i+1}},\ldots, u^d )\in \bR^d.$$
For any $\gamma\in (0, 1]$, we consider the condition 
\begin{itemize}
\item [$\bf(D_\gamma)$: ] For any $x, y\in U$ and for any permutation  $\{i_1, i_2, \ldots, i_d \}$ of $\{1, 2, \ldots, d\}$, let 	$\xi_{(0)},\xi_{(1)},\dots,\xi_{(d)}$ be points that trace from $x$ to $y$ with
$$\xi_{(0)}:=x,\  \xi_{(1)}:=[x]_{y^{i_1}}^{i_1},\ \xi_{(2)}:=[\xi_{(1)}]_{y^{i_2}}^{i_2},\ldots, \xi_{(d-1)}:=[\xi_{(d-2)}]_{y^{i_{d-1}}}^{i_{d-1}}\text{ and }\xi_{(d)}:=[\xi_{(d-1)}]_{y^{i_d}}^{i_d}=y.$$
\end{itemize}
We say that $U$ satisfies condition $\bf (D_\gamma)$ if for any $r>0$ and $x, y\in U$ with $\delta_U(x)\wedge \delta_U(y)\ge r$, there exists a permutation  $\{i_1, i_2, \ldots, i_d\}$ of $ \{1, 2, \ldots, d\}$ such that 
$${B( \xi_{(k)}, \gamma r)}\subset U,\qquad \text{ for all }k=1,2,\ldots, d.$$
By the definition of $\bf (D_\gamma)$, one can easily see that 
\begin{equation}\label{eq:rel}
	\xi_{(l)}^{i_k}=\begin{cases}
		y^{i_k}&\quad\text{ if }l\ge k\\
		x^{i_k}&\quad\text{ if }l< k.
	\end{cases}
\end{equation} 

\begin{theorem}\label{t:nmain2}
	Under the setting of \autoref{t:nmain}, assume in addition that $D$ satisfies $(\mathbf D_\gamma)$ for some $\gamma\in(0,1]$.
	\begin{enumerate}[\rm(1)]
		\item Then for any $T>0$, there exists $C_3:=C_3(\phi, \kappa_0, \kappa_1,  \eta, R, \Lambda,  T, d, \gamma)>0$ such that for any $t\in(0,T]$ and $x,y\in D$,
		\begin{equation}\label{eq:dhke_l}
			p_D(t,x,y)\ge C_3\Psi(t,x)\Psi(t,y)p(t,x,y).
		\end{equation}
		\item Assume that $D$ is bounded, then there exists $C_4:=C_4(\phi, \kappa_0, \kappa_1,  \eta, R, \Lambda,  d, \gamma, \text{diam}(D))>1$ such that
		\begin{align}\label{eq:La1}
			C_4^{-1}\le \lambda^D\le C_4,
		\end{align}
		where $-\lambda^D<0$ is the first eigenvalue of $-\cL$ on $D$.
		\medskip
		
		\noindent	Also for any $T>0$, there exists $C_5=C_5(\phi, \kappa_0, \kappa_1,  \eta, R, \Lambda,  T, d, \gamma, \text{diam}(D))> 1$ such that for all $t\in [T, \infty)$ and $x, y\in D$, 
		\begin{align}\label{eq:La2}
			C_5^{-1} e^{-t \lambda^D } \sqrt{\phi\left(\delta_D(x)\right)}\sqrt{\phi\left(\delta_D(y)\right)}\le p_D(t, x, y)\le C_5 e^{-t \lambda^D } \sqrt{\phi\left(\delta_D(x)\right)}\sqrt{\phi\left(\delta_D(y)\right)}.
		\end{align}
	\end{enumerate}
\end{theorem}

\medskip
We now consider $x,y\in D$ with $x^i\ne y^i$ for all $1\le i\le d$. By integrating the heat kernel estimates in  \autoref{t:nmain} and \autoref{t:nmain2} with respect to $t\in(0,\infty)$, one gets the following Green function estimates.

\begin{theorem}\label{T:green_new}
	Suppose that 
	$X$ is a symmetric pure jump Hunt process whose jumping intensity kernel $J$ satisfies the conditions \eqref{a:J}, {\bf (SD)} and {\bf (K$_\eta$)}. 
	Suppose that $D$ is a bounded $C^{1,1}$ open set in $\bR^d$ with characteristics $(R, \Lambda)$  under the condition $(\bf D_\gamma)$.
	For any $(x, y)\in D\times D$, suppose that $x^i\ne y^i$ for all $i\in\{1, \ldots, d\}$. 
	Then  when $d\ge 2$, there exists $C_6>1$ such that 
	\begin{align*}
		&	C_6^{-1} 
		\Big(1\wedge\frac{\sqrt{\phi(\delta_D(x))}}{\sqrt{\phi( r_1(x,y))}}\Big)\Big(1\wedge\frac{\sqrt{\phi(\delta_D(y))}}{\sqrt{\phi( r_1(x,y))}}\Big)\phi( r_1(x,y))^{d+1}\prod_{l=1}^d\frac{1}{|x^{i}-y^{i}|\phi(|x^{i}-y^{i}|)}\\
		\le 	&\	G_D(x, y)\\
		\le	&\ C_6 \left(
		\Big(1\wedge\frac{\sqrt{\phi(\delta_D(x))}}{\sqrt{\phi(r_2(x, y))}}\Big)\Big(1\wedge\frac{\sqrt{\phi(\delta_D(y))}}{\sqrt{\phi(r_2(x, y))}}\Big)	\phi(r_2(x, y))^{d+1}\right)\prod_{l=1}^d\frac{1}{|x^{i}-y^{i}|\phi(|x^{i}-y^{i}|)}.
	\end{align*}
	where $r_1(x,y):=\min_{i\in\{1,\dots,d\}}|x^i-y^i|,r_2(x,y):=\max_{i\in\{1,\dots,d\}}|x^i-y^i|$.
	When $d=1$,
	letting $a(x, y):=\sqrt{\phi(\delta_D(x))}\sqrt{\phi(\delta_D(y))}$ and $x^+:=\max (x, 0)$, we have that 
	\begin{align*}
		G_D(x, y)\asymp \displaystyle	\frac{a(x, y)}{|x-y|}\wedge \left(\frac{a(x, y)}{\phi^{-1}(a(x, y))}+\Big(\int_{|x-y|}^{\phi^{-1}(a(x, y))}\frac{\phi(s)}{s^2}ds\Big)^+\right).
	\end{align*}
\end{theorem}

\medskip

The rest of this article is organized as follows. In \autoref{sec:analandexit}, we compute some key upper bounds of the generator of $Z$ with testing functions defined on $C^{1,1}$ open sets, and discuss the key estimates on exit distributions for $X$, see \autoref{L:exit}. In \autoref{sec:regupper}, we first discuss the regularity of the transition density $p_D(t,x,y)$ of $X^D$, that is, $p_D(t,x,y)$ can be refined to be H\"older continuous for $(t,x,y)\in \bR_+\times D\times D$. Then we adopt the method in \cite[Section 5]{GKK20} to obtain the upper bound of $p_D(t,x,y)$ stated in \autoref{t:nmain}. However, a variation of the key lemma \cite[Lemma 5.1]{GKK20} is needed since the process here is anisotropic (see, \autoref{lemma:pU}). In \autoref{sec:lower}, we prove the lower bound of $p_D(t,x,y)$ in \autoref{t:nmain2}(1), by utilizing a preliminary lower bound result in \autoref{prp:lower1} as well as the estimates of exit distributions presented in \autoref{L:exit}. The large time Dirichlet heat kernel estimates in \autoref{t:nmain2}(2) are proved in \autoref{sec:largeDHKE}.
In \autoref{sec:green}, as an application of \autoref{t:nmain} and \autoref{t:nmain2}, we derive the Green function estimates. We would like to mention that this is  the very first attempt to observe the Green function estimates of anisotropic Markov processes, though only for the points with different coordinates in all directions.
\medskip

{\bf Notations.}
For  a function space ${\mathbb F}(U)$
on an open set $U$ in $\bR^d$, we let   
\begin{align*}
	{\mathbb F}_c(U):=&\{f\in{\mathbb F}(U): f \mbox{ has  compact support}\}.
\end{align*}
We use $\langle \cdot, \cdot\rangle$  to denote the inner product in $L^2(\Rd)$ and $\|\cdot\|_k:=\|\cdot\|_{L^k(\bR^d)}$.
For $x\in\bR^d$ and $r>0$, we use $B(x, r)$ to denote  a ball centered at $x$ with  radius $r$, and $Q(x,r)$ to denote  a cube centered at $x$ with  side length $r$.
The letter $c=c(a, b, \ldots)$ will denote a positive constant depending on $a,b,\ldots$, and it may change at each appearance.
The labeling of the constants $c_1, c_2, \ldots$ begins anew in the proof of each statement.
The notation $:=$ is to be read as ``is defined to be".
We use the notation $f\asymp g$ if the quotient $f/g$ is comparable to some positive constants.

\section{Analysis on $Z$ and Exit distribution for $X$}\label{sec:analandexit}
{Recall that $Z=(Z^1,\dots,Z^d)$ is the $d$-dimensional L\'evy process where $Z^i$'s are independent 1-dimensional L\'evy processes with the jumping kernel $J^{\phi}(x,y) =|x^i-y^i|^{-1}\phi(|x^i-y^i|)^{-1}$, where $\phi$ is an increasing function satisfying ({\bf WS}) and ({\bf SD}). $X=(X^1,\dots,X^d)$ is the anisotropic Markov process with the jumping kernel $J$ satisfying  ({\bf SD}) and ($\mathbf K_\eta$). 
	
	\medskip
	In \autoref{subsec:analysis}, we introduce the renewal function $V$ corresponding to the coordinate process $Z^d$, and obtain the estimates for the generator of $Z$ applied to the testing functions on $C^{1,1}$ open sets, see \autoref{g:st1}. In \autoref{subsec:exit}, we give the key estimates on exit distributions for $X$, see \autoref{L:exit}.}
\subsection{Analysis on $Z$}\label{subsec:analysis}

Recall that $Z^{d}$ is a pure jump L\'{e}vy process with L\'{e}vy measure $\nu^1(dr)$. 
The L\'evy-Khintchine (characteristic) exponent of $Z^{d}$ has the form
\begin{equation*}
	\psi(|\xi|)=\int_\bR \left(1- \cos (\xi r) \right) \nu^1(dr) \qquad\text{ for } \xi\in\bR.
\end{equation*}

Let $M_t:=\sup_{s\le t}Z^{d}_s$ and  $L$ be the local time at $0$ for $M-Z^{d}$, the reflected process of $Z^{d}$ from the supremum.
For the right-continuous inverse process $L^{-1}$, define the ascending ladder-height process $H_s:= Z^{d}_{L^{-1}_s} = M_{L^{-1}_s}$.
Then the Laplace exponent of $H_s$ is
\begin{equation*}
	\chi(\xi)= \exp\left(\frac{1}{\pi} \int_0^\infty \frac{ \log (\psi(\theta \xi))}{1 + \theta^2} \, d\theta\right) \qquad \text{ for }\xi\ge 0
\end{equation*}
(see, \cite[Corollary 9.7]{MR0400406}).
Define the renewal function $V$ of $H$ as follows:
\begin{equation*}\label{e:defV}
	V(x):= \int_0^{\infty}\Pp(H_s \le x)ds \quad\text{ for }
	x \in \bR.
\end{equation*}
Then $V(x)=0$ for $x<0$, $V(\infty)=\infty$ and $V$ is non-decreasing.  Also $V$ is sub-additive (see \cite[p.74]{MR1406564}), that is,
\begin{equation}\label{subad}
	V(x+y)\le V(x)+V(y) \quad\text{ for } x,y \in \bR.
\end{equation}
Since the distribution of $Z^{d}_t$ is absolutely continuous for every $t>0$, the resolvent measure of $Z^{d}_t$ is absolutely continuous as well, see, \cite[Theorem 6]{MR0341626}.
From \cite[Theorem 2]{MR573292} and \cite[(1.8) and Theorem 1]{MR573292},
we see that $V(x)$ is absolutely continuous and harmonic on $(0,\infty)$ for $Z^{d}_t$, and $V'$ is a positive harmonic function  on $(0,\infty)$ for $Z^{d}_t$. Therefore, $V$ is actually (strictly) increasing.

\medskip

Corresponding to the L\'evy measure $\nu^1(dz)$, the Pruitt function (see \cite{MR632968}) is defined as follows:
$$h(r):=\int_{\bR} \left(1 \wedge {|z|^2}r^{-2}\right)\nu^1(dz) \qquad \mbox{ for  }\,\,r>0,$$
and by \cite[Corollary 3]{BGR14} and \cite[Proposition 2.4]{BGR15}, 
\begin{equation*}\label{compVhp}
	h(r)\asymp [V(r)]^{-2}\asymp \psi(r^{-1})\qquad\mbox{ for }\, r>0.
\end{equation*}
Since $s\to \phi(s^{-1})^{-1}$ satisfies  ({\bf WS}), \eqref{d:nu1} and \cite[Proposition 28]{BGR14} yield that
$$\psi(r)\asymp \phi(r^{-1})^{-1}\qquad \mbox{ for }\, r>0.$$
Combining {the above} observations, we conclude that 
\begin{align}\label{e:com}
	\phi(r)\asymp [V(r)]^2 \qquad \mbox{ and } \qquad 
	\nu^1(r) \asymp [V(r)]^{-2} r^{-1}
	\quad \mbox{ for  } \,\, r>0.
\end{align}
Also ({\bf WS}) condition of $\phi$ yields {the weakly scaling properties} of $V$ as follows: there exists $C_V>1$ such that 
\begin{align}\label{V:sc}
	C_V^{-1} \left(\frac{R}{r}\right)^{\la/2}\le \frac{V(R)}{V(r)}\le C_V \left(\frac{R}{r}\right)^{\ua/2}\qquad \mbox{ for any } 0<r\le R.
\end{align}
By \eqref{e:com} and \eqref{V:sc}, it is clear that
for any $r>0$, there is $c=c(\la, \lC)>0$ such that
\begin{align}\label{eq:Vnu}
	\int_{\{|t|>r\}} V(|t|)\nu^1(|t|)dt\le \frac{c}{V(r)}\quad \text{ and }\quad
	\int_{\{|t|>r\}} \nu^1(|t|)dt\le \frac{c}{[V(r)]^2}.
\end{align}

\medskip

The following proposition is directly obtained by \cite[Theorem 1.1]{TKMR16} 
since $V$ and $V'$ are harmonic on $(0, \infty)$ for $Z^{d}$, and the assumption \cite[{\bf(A)}]{TKMR16} is satisfied by ({\bf SD}).

\begin{proposition}\label{VbisEst} 
	The function $x \to V(x)$ is twice-differentiable  for any $x>0$, and there exists a positive constant $C_{\ref{VbisEst}}$ such that 
	\begin{align*}
		V'(x)\leq C_{\ref{VbisEst}}\frac{V(x)}{x\wedge 1}\qquad \text{and}\qquad
		|V''(x)|\leq C_{\ref{VbisEst}}\frac{V'(x)}{x\wedge 1}\quad \mbox{ for any } x>0.
	\end{align*}
\end{proposition}

\medskip

For any $x=(\wt x, x^d)\in \Rd$, define $w(x)=V(x^d\vee 0)$. 
Then we have the proposition which follows from \cite[Proposition 3.2] {GKK20}. For the reader's convenience, we give the proof.
\begin{proposition}\label{st2}
	For $\lambda>0$, there exists $C_{\ref{st2}}=C_{\ref{st2}}(\lambda, \phi)>0$ such that for any $r>0$  and $i\in\{1,\dots,d\}$, we have
	\begin{align*}
		\sup_{\{x\in \Rd\, :\, 0<x^d\le \lambda r\}}\int_{\{|t|>r\}} w(x+e_it)\nu^1(|t|)dt<  \frac{C_{\ref{st2}}}{V(r)}.
	\end{align*}
\end{proposition}

\begin{proof}\hspace{-.05in}.
	Since $w(x+z)\le V(x^d)+V(|z|)$ for $x^d>0$ and $z\in \Rd$, it follows that
	\begin{align*}
		\int_{\{|t|>r\}} w(x+e_it)\nu^1(|t|)dt\leq V(x^d)\int_{\{|t|>r\}}\nu^1(|t|)dt+\int_{\{|t|>r\}} V(|t|)\nu^1(|t|)dt.
	\end{align*}
	By \eqref{eq:Vnu} and \eqref{subad},  we conclude that 
	$$\sup_{\{x\in \Rd\, :\, 0<x^d\le \lambda r\}}\int_{\{|t|>r\}} w(x+e_it)\nu^1(|t|)dt\leq c_1 \left(\frac{V(\lambda r)}{[V(r)]^2}+\frac{1}{V(r)}\right)\le \frac{c_2}{V(r)}.$$
\end{proof}

For each $i\in\{1, \ldots, d\}$,
we define an operator $\cL^{(i)}_Z$ by
\begin{align}\label{e:gzi}
	\cL^{(i)}_Z f(x):=\lim_{\varepsilon\downarrow 0}\int_{\{|t|>\eps\}} (f(x+e_i t)-f(x))\nu^1(|t|)dt=:\lim_{\varepsilon\downarrow 0}
	\cL^{(i), \eps}_Z f(x)
	\quad\text{ for } x\in \Rd.
\end{align}
Applying \cite[Theorem 3.3]{GKK20} with  \autoref{VbisEst} and  \autoref{st2}, we have 
the following Theorem.

\begin{theorem}\label{std}
	For any {$x=(\wt x,x^d)\in\bR^d$ with $x^d>0$},  $\cL_Z^{(i)} w(x)$ is well-defined and $\cL_Z^{(i)} w(x)=0$ for any $i\in \{1,\dots,d\}$.
\end{theorem}

\medskip
For the rest of this paper, let $D$ be a $C^{1,1}$ open set in $\bR^d$ with characteristics $(R,\Lambda)$. 
We may assume that  $R<1$ and $\Lambda >2$. For $Q\in\partial D$, consider the coordinate system  $CS_Q$ such that
$$B(Q, R)\cap D=\{(\wt y, y^d)\in B(Q, R) \text{ in } CS_Q: y^d>\varphi_Q(\wt y)\}.$$
Define a function $\rho_{Q}(y):=y^{d}-\varphi_Q(\wt{y})$ in $CS_Q$.
Then 
\begin{align}\label{comp:rho}
	\frac{\rho_Q(y)}{\sqrt{1+\Lambda^2}}\le \delta_D(y)\le \rho_Q(y)\qquad \text{ for any }y\in B(Q ,R)\cap D. 
\end{align}
For $r_1, r_2>0$, let
\begin{align*}
	D_Q(r_1, r_2):=\{y\in D:|\wt{y}| <r_1,\ 0<\rho_Q(y)<r_2\}.
\end{align*}

\medskip

\begin{proposition}\label{g:st1}
	For any $Q\in \partial D$ and $R_0\le  R/2$, we define 
	$$h_{R_0}(y)=h_{Q, R_0}(y):=V(\delta_D(y))\1_{D\cap B(Q, R_0)} (y)\qquad \text{ for }\ y\in \Rd.$$
	Let $r_0:={R_0}/4\Lambda$.
	Then there exists a constant $c=c(\phi,  R, \Lambda)>0$ independent of $Q$ such that for any $i\in\{1,\dots,d\}$, $\cL_{Z}^{(i)} h_{R_0}$ is well-defined in $D_Q(r_0, r_0)$ and 
	\begin{align*}
		|\cL_{Z}^{(i)} h_{R_0}(x)|\le \frac{c}{V(R_0)}\,\qquad\mbox{ for any }x\in D_Q(r_0, r_0). 
	\end{align*}
\end{proposition}
\begin{proof}\hspace{-.05in}.
	For any $x\in D_Q(r_0, r_0)$, let $x_0\in \partial D$ satisfy $\delta_D(x)=|x-x_0|$.
	Let
	$\varphi:=\varphi_{x_0}$ be the $C^{1,1}$ function and $CS:=CS_{x_0}$ be an orthogonal coordinate system with $x_0$ so that 
	$$\varphi(\widetilde 0)=0\,,\ \nabla\varphi(\widetilde 0)=(0, \dots, 0)\,,\ \| \nabla \varphi \|_\infty \leq \Lambda\,,\  |\nabla \varphi (\widetilde y)-\nabla \varphi (\widetilde z)| \leq \Lambda |\widetilde y-\widetilde z|$$
	$$x=(\widetilde{0}, x^d)\  \text{ and }\
	D\cap B({x_0}, R)=\{y=(\widetilde{y}, y^d)  \in B(0, R) \mbox{ in } CS: y^d>\varphi(\widetilde{y})\}.$$
	We fix the function $\varphi$ and the coordinate system $CS$, and we define a function 
	$$g_x(y)=V(\delta_{{\mathbb H}}(y))\1_{D\cap B(Q, R_0)} (y)=V(y^d)\1_{D\cap B(Q, R_0) }(y)\quad \text { for } y\in \bH $$ 
	where ${\mathbb H}=\{y=(\widetilde{y}, y^d) \mbox{  in } CS :y^d>0\}$ is the half space in $CS$. 
	Note that  $ h_{R_0}(x)=g_x(x)=V(x^d)$ and  that $ \cL^{(i)}_Z (h_{R_0}- g_x)=\cL^{(i)}_Z h_{R_0}$ by \autoref{std}. Therefore it suffices to show that $\cL^{(i)}_Z (h_{R_0}- g_x)$ is well defined, and
	that there exists a constant $c_1=c_1(\phi,  R, \Lambda)>0$ such that 
	\begin{align}\label{e:claims}
		\int_{\{t\in\bR : x+e_i t\in (D\cup {\mathbb H})\}}|h_{R_0}(x+e_i t)-g_x(x+e_i t)|\nu^1(|t|) dt\le c_1 V(R_0)^{-1}.
	\end{align}	
	
	Define $\widehat \varphi :B(\widetilde{0}, R_0)\to \bR$ by $\widehat \varphi(\widetilde z):=2\Lambda|\widetilde z|^2$.
	Since $\varphi (\widetilde{0})=0$,  the mean value theorem implies that 
	$|\varphi(\widetilde{z})|\le \Lambda|\wt z|^2\le \widehat \varphi(\widetilde{z})$  for any ${z=(\wt z,z^d)}\in D\cup \bH$.
	Let 
	$$A:=\{  z\in D\cup {\mathbb H}: |\wt z|<r_0,\ |z^d|\le \widehat{\varphi}(\widetilde{z})\}\ \text{ and }\ E:=\{ z\in D:|\wt z|<r_0,\ \widehat{\varphi}(\widetilde{ z})<z^d< 3r_0\Lambda \}.$$
	Then for any $z\in D\cap B(x, r_0)\backslash A$,   since  $x\in D_Q(r_0, r_0)$,
	\begin{align*}
	 \widehat{\varphi}(\widetilde{ z})<	z^d\le |z^d-x^d|+|x^d|\le  r_0+\delta_D(x)\le r_0+ \sqrt{2}r_0\le 3r_0,
	\end{align*}
	so that $z\in E$.
	Also
	for any $z\in E$, since $ Q\in \partial D$ and $x_0\in D_Q(r_0, r_0)$,
	\begin{align*}
		|z-Q|&\le |z-x_0|+|x_0-Q|\le \sqrt{|z^d|^2+|\wt z|^2}+\sqrt{2} r_0\le\sqrt{ (3r_0\Lambda)^2+r_0^2}+\sqrt{2}  r_0\le4\Lambda r_0=
		R_0
	\end{align*}
	so that $z\in D\cap B(Q, R_0)$. Therefore, 
	$$D\cap B(x, r_0)\backslash A\subset E\subset D\cap B(Q, R_0).$$
	Now we decompose \eqref{e:claims} as follows:
	\begin{align*}
		&\int_{\{|t|\ge r_0\}}|h_{R_0}(x+e_it)-g_x(x+e_it)|\nu^1(|t|) dt
		+\int_{\{|t|<r_0\}}
		|h_{R_0}(x+e_it)-g_x(x+e_it)|\nu^1(|t|) dt\\
		\le &\int_{\{|t|\ge r_0\}}(h_{R_0}(x+e_it)+g_x(x+e_it))\nu^1(|t|) dt+\int_{\left\{x+e_it\in A\right\}}
		(h_{R_0}(x+e_it)+g_x(x+e_it))\nu^1(|t|) dt\nn\\
		&+\int_{\{x+e_it\in E\}}
		|h_{R_0}(x+e_it)-g_x(x+e_it)|\nu^1(|t|) dt=:\ \I+\II+\III.
	\end{align*}	
For any fixed $x\in D_Q(r_0, r_0)$, we denote $y_t:=x+e_i t, t>0$.

In $CS$, let $\Pi$ be the hyperplane tangent to $\partial D$ at $x_0$. Consider a function $\Gamma:\bR^{d-1}\to \bR$ defined by $\Gamma(\wt y):=\nabla \varphi(\wt x)\cdot \wt y$ which describes the plane $\Pi$. 
Let $\nabla \varphi(\wt x):=(\alpha^1, \ldots, \alpha^d)$ 
and  $ \theta_1$ be the angle between the normal vector $\nabla \varphi(\wt x)$ of the hyperplane $\Pi$  and $y_t-x=e_it$. Then  since
$\alpha^it=\langle \nabla \varphi(\wt x), e_i t\rangle
=|\nabla \varphi(\wt x)||t| \cos \theta_1$, $|\cos \theta_1|=\frac{|\alpha^i|}{|\nabla \varphi(\wt x)|}$ and
\begin{align}\label{e:new}
	|\wt y_t|=|t\sin\theta_1|\qquad\text{and }\qquad |y^d-x^d|=|t\cos\theta_1 | 
\end{align}

	Now we estimate $\I, \II$ and $\III$ as follows:\\
	{$\bf I$:}
	Let $|t|\ge r_0$.
	Since $R_0={4\Lambda r_0\leq 4\Lambda|t|}$,
	by the definition of $h_{R_0}$ with \eqref{subad}, \eqref{eq:Vnu} and \autoref{st2}, we have
	\begin{align}\label{e:I}
		\I&\le  c_2	\int_{\{|t|\ge r_0\}} V(|t|)\nu^1(|t|)dt +\sup_{x\in\Rd:0<x^d<r_0}\int_{\{|t| \ge r_0\}} g_x(y_t) \nu^1(|t|)dt\\
		&\le  c_3 V(r_0)^{-1}+\sup_{x\in\Rd:0<x^d<r_0}\int_{\{|t| \ge r_0\}} w(x+e_it)\nu^1(|t|)dt\le (c_3 +C_{\ref{st2}})V(r_0)^{-1}.	\nn
	\end{align}
{$\bf II$:}
For $y_t=x+e_i t\in A$, that is, $\theta_1\in (0, \pi)$, hence $|\wt y_t|\neq 0$. Since $|y_t^d|\vee \delta_D(y_t)\le 2|\widehat\varphi(\wt y_t)|
	\le 4\Lambda |\wt y_t|\le 4\Lambda |t|$ by \eqref{e:new} and  $V$ is increasing,
	$$h_{R_0}(y_t)+g_x(y_t)\le c_4 V(|\wt y_t|) \le c_4V(|t|).$$
	For the surface measure $m_{d-1}(dy)$ on $\bR^{d-1}$,
	$ m_{d-1}\big(\{y_t\in \bR^d: |\wt y_t|= s |\sin\theta_1|,  |y_t^d|\le 2\Lambda |\wt y_t|^2\}\big)\le c_5( s|\sin\theta_1|)^d\le c_5s^d$. Then by \eqref{e:com} and \eqref{V:sc}
	\begin{align}\label{e:II}
		\II&\le    c_4\int_{\left\{y_t\in A\right\}}
		V(| \wt y_t|)\nu^1(|t|) dt
		\le  c_6  \int_{0}^{r_0} \frac{V(s)}{[V(s)]^2s}\cdot  s^d ds
		= c_6
		\int_{0}^{r_0} \frac{s^{d-1}}{V(s)}ds  \nn\\
		&
		\le c_7\frac{r_0^{\ua/2}}{V(r_0)}\int_{0}^{r_0} s^{d-1-\frac{\ua}{2}}ds
		\le \frac{c_7}{(d-\frac{\ua}{2})}\frac{1}{V(r_0)}.
	\end{align}
	{$\bf III$:}
	For $y_t\in E\subset D$, 
	we first note that
	\begin{align}\label{eq:dis}
		|y_t^d-\delta_D(y_t)|\le c_8|\wt y_t|^2.
	\end{align}
To obtain this, we consider the cases $y_t^d=\delta_{{\mathbb H}}(y_t)\le \delta_D(y_t)$  and $\delta_D(y_t)\le y_t^d=\delta_{{\mathbb H}}(y_t)$  separately.
	If $0<y_t^d=\delta_{{\mathbb H}}(y_t)\le \delta_D(y_t)$,  by the outer ball condition
	\begin{align*}
		\delta_D(y_t)-y_t^d&\le \Big(\sqrt{|\wt y_t|^2+(R+y_t^d)^2}-R\Big)-y_t^d=\frac{|\wt y_t|^2}{\sqrt{|\wt y_t|^2+(R+y_t^d)^2}+(R+y_t^d)}\le \frac{|\wt y_t|^2}{R}.
	\end{align*}
	Similarly, if $y_t^d=\delta_{{\mathbb H}}(y_t)\ge \delta_D(y_t)$,
	the interior ball condition implies that
	\begin{align*}
		y_t^d- \delta_D(y_t)&\le y_t^d -\Big(R-\sqrt{|\wt y_t|^2+(R-y_t^d)^2}\Big)=\frac{|\wt y_t|^2}{\sqrt{|\wt y_t|^2+(R-y_t^d)^2}+(R-y_t^d)}\le \frac{|\wt y_t|^2}{R/2}.
	\end{align*}
 For $|t|\le r_0$, by \eqref{e:new}, since
$|\wt y_t|\le r_0$ and $|y_t^d|\le |x^d|+|t|\le  2r_0$, 
	$$\delta_D(y_t)^2\le |y_t-x_0|^2= |\wt y_t|^2+|y_t^d|^2\le  5r_0^2.$$
Using the above observation with the scale invariant Harnack inequality for $Z^{d}$(see, \cite[Theorem 1.4]{CKK09}) applied to $V^{\prime}$,  \autoref{VbisEst} and \eqref{V:sc}, we have that
	for $|t|\le r_0$,
	\begin{align}\label{e:E}
		\sup_{u\in [y_t^d\wedge \delta_D(y_t), \ y_t^d\vee \delta_D(y_t)]} V'(u)
		&\le \sup_{u\in [y_t^d\wedge \delta_D(y_t),  \sqrt 5 r_0]} V'(u)
		\le  c_9 \inf_{u\in [y_t^d\wedge \delta_D(y_t),  \sqrt 5 r_0]} V'(u)\nn\\
		&\le c_9 V'(\sqrt 5 r_0)\le c_{10}{V(r_0)}/{r_0}.
	\end{align}
Combining \eqref{eq:dis}--\eqref{e:E}, the mean value theorem yields that  for $y_t\in E$,
	\begin{align*}
		|h_{R_0}(y_t)-g_x(y_t)|\le 
		\sup_{u\in [y_t^d\wedge \delta_D(y_t), \,\, y_t^d\vee \delta_D(y_t)]} V'(u)| y_t^d-\delta_D(y_t)|\le c_{11} V\left(r_0\right) |t|^2/r_0.
	\end{align*}
Therefore, by \eqref{e:com}--\eqref{V:sc} with the fact that $r_0<1$, 
\begin{align}\label{e:III}
\III \le&\  c_{11}\frac{V(r_0)}{r_0}\int_0^{r_0} \frac{s}{[V(s)]^2}ds
\le  \frac{c_{12} r_0^{\ua-1}}{V(r_0)}\int_0^{r_0} s^{1-\ua} ds
\le \frac{c_{13}r_0}{V(r_0)}\le \frac{c_{13}}{V(r_0)}.
\end{align}
Hence, we conclude our claim by \eqref{e:I}, \eqref{e:II} and \eqref{e:III} using the definition of $r_0$ and \eqref{V:sc}.
\end{proof}

\medskip
\subsection{Estimates on exit distribution for $X$}\label{subsec:exit}
{Recall that $X$ is a conservative Feller process with the jumping kernel $J(x,y)$ defined in \eqref{a:J} satisfying the conditions  ({\bf SD}) and ($\mathbf K_\eta$).} For each $i\in \{1, \ldots, d\}$, define operators $\cL^{(i)}$ and $\cL$ on $C^2_c(\Rd)$ as
\begin{align*}
	\cL^{(i)} f(x):=
	\lim_{\eps\downarrow 0}
	\int_{\{|t|>\eps\}} \big(f(x+e_i t)-f(x)\big)J(x, x+e_i t)dt
	=: \lim_{\varepsilon\downarrow 0}\cL^{(i)}_{\eps}f(x)
\end{align*}
and $\cL f(x):=\sum_{i=1}^d\cL^{(i)} f(x)$.
For any $i\in\{1,\dots,d\}$, $g\in C_c^2(\Rd)$ and $0<\varepsilon<r<1$,  we have
\begin{align*}
	\cL^{(i)}_{\eps} g(x)
	& = \kappa(x,x)\int_{\{\varepsilon<|t|<r\}} \big(g(x+e_i t)-g(x)-(e_it)\cdot\nabla g(x)\big)\nu^1(|t|)dt\\
	&+\int_{\{\varepsilon<|t|<r\}}\big(g(x+e_i t)-g(x)\big)\big(\kappa(x,x+e_i t)-\kappa(x,x)\big)\nu^1(|t|)dt\\
	&+\int_{\{r\le|t|\}}\big(g(x+e_i t)-g(x)\big)J(x, x+e_i t)dt.
\end{align*}
For $x\in \Rd$ and $|t|\le1$, by \textbf{(K$_{\eta}$)}, we first obtain that
\begin{align*}
	|(g(x+e_i t)-g(x))(\kappa(x,x+e_i t)-\kappa(x,x))|&\le ||g||_{\infty}\kappa_1|t|^{\eta+1}.
\end{align*}
Also ({\bf WS}) with the fact that $ 1\ge \eta> \ua/2> \ua-1$ implies that for $0<r<1$,
\begin{align}\label{eq:nub}
	\int_{\{|t|<r\}}|t|^{2} \nu^1(|t|)dt\le \int_{\{|t|<r\}}|t|^{\eta+1} \nu^1(|t|)dt<\infty.
\end{align}
Therefore, $\cL^{(i)} g$ is well defined and $\cL^{(i)}_{\eps} g$ converges to $\cL^{(i)} g$ uniformly on $\Rd$.
Furthermore, for every $0<r<1$,
\begin{align}\label{eq:cLi}
	\cL^{(i)} g(x)&=\kappa(x,x)\int_{\{|t|<r\}}\big(g(x+e_i t)-g(x)-(e_it)\cdot\nabla g(x)\big)\nu^1(|t|)dt\nn\\
	&+\int_{\{|t|<r\}}\big(g(x+e_i t)-g(x)\big)\big(\kappa(x,x+e_i t)-\kappa(x,x)\big)\nu^1(|t|)dt\\
	&+\int_{\{r\le|t|\}}\big(g(x+e_i t)-g(x)\big)J(x, x+e_i t)dt.\nn
\end{align}

The following proof is almost the same as that of \cite[Lemma 2.2]{GKK20}.

\begin{lemma}\label{Dyk_F}
	For any $g\in C^2_c(\Rd)$ and $x\in \bR^d$, there exists a $\Pp_x$-martingale $M^g_t$ with respect to the filtration of $X$ such that 
	$$M^g_t=g(X_t)-g(X_0)-\int_0^{t}  {\cL} g(X_s) ds\quad \text{ 	$\Pp_x$-a.s.} $$
	In particular, for any stopping time $S$ with $\E_x [S]<\infty$,
	we have
	\begin{align}\label{e:SnS}
		\E_x \left[ g(X_{S})\right]-g(x)=\E_x\left[\int^{S}_0\cL g(X_s)ds\right].
	\end{align}
\end{lemma}

\begin{proof}\hspace{-.05in}.
	Let $(\cA, \cD(\cA))$ be the $L^2$-generator of the semigroup $T_t$ with respect to $X$.
	By the definition of $\cL$, we could apply the similar proof as that of \cite[Proposition 2.5]{MR2365348} to obtain that $C_c^2(\Rd) \subset \cD(\cA)$ and $\cA|_{C_c^2(\Rd)}=\cL |_{C_c^2(\Rd)}$.
	Since $T_t$ is strongly continuous (see e.g., \cite[Theorem 1.3.1 and Lemma 1.3.2]{MR569058}),
	for any $f\in \cD(\cA)$ and $t\geq 0$,
	$$\left|\left|(T_tf-f)-\int^t_0T_s\cA f\ ds\right|\right|_{L^2}=0$$
	(see e.g. \cite[Proposition 1.5]{MR838085}). 
	Hence for $g\in C_c^2(\Rd)$, 
	\begin{equation}\label{DF1}
		T_tg(x)-g(x)=\int^t_0T_s \cL g(x)ds,\quad a.e.\,\, x\in \Rd.
	\end{equation}
	{Note that by \eqref{d:nu1}, \eqref{a:kappa}, \textbf{(K$_{\eta}$)} and \eqref{e:large}, the absolute value of \eqref{eq:cLi} is bounded as follows:
		\begin{align}\label{e:Li}
			|\cL^{(i)} g(x)|\le  c_1\Big(\kappa_0 ||\partial^2 g||_\infty \int_0^r \frac{s}{\phi(s)} ds &+\kappa_1||\nabla g||_\infty\int_0^r \frac{s^{\eta}}{\phi(s)} ds\Big) + 2 ||g||_\infty\frac{c_2}{\phi(r)}.
		\end{align}
		Similar as \eqref{eq:nub}, by ({\bf WS}) with the fact that $\eta>  \ua/2> \ua-1$, we have that
		\begin{align*}
			&\int_0^r \frac{s}{\phi(s)} ds\le c_3 \frac{r^2}{\phi(r)}\quad \text{ and }	\quad	\int_0^r \frac{s^{\eta}}{\phi(s)} ds\le c_4 \frac{r^{\eta+1}}{\phi(r)}.
		\end{align*}
		Applying these inequalities to  \eqref{e:Li}, we conclude that there exists a constant $c=c(\phi,  \eta, \kappa_0, \kappa_1)>0$ such that for any
		function  $g\in C_c^2(\Rd)$ and $0<r<1$,
		\begin{equation}\label{cLig}
			||\cL^{(i)} g||_\infty\leq \frac{c}{\phi(r)}\left( r^2||\partial^2 g||_\infty+r^{\eta+1}||\nabla g||_\infty+||g||_\infty\right).
		\end{equation} 
		Thus, $\cL g=\sum_{i=1}^d\cL^{(i)} g(x)$ is bounded.}

	Let ${\cH}_t(x):=\int^t_0T_s\cL g(x)ds$,
	then ${\cH}_t(x)\leq t||\cL g||_\infty$ which is bounded for fixed $t>0$.
	Since $T_\varepsilon$, $\varepsilon>0$, is strong Feller, 
	$T_\varepsilon {\cH}_{t-\varepsilon}\in C_b(\Rd)$, $\varepsilon\in (0, t)$. Also since	$$|{\cH}_t(x)-T_\varepsilon {\cH}_{t-\varepsilon}(x)|=|{\cH}_\varepsilon(x)|\leq \varepsilon ||\cL u||_\infty,$$
	${\cH}_t$ is continuous and  therefore \eqref{DF1} holds for any $x\in\Rd$. 
	Therefore, Markov property implies that
	$$M^g_t=g(X_t)-g(X_0)-\int^t_0\cL g(X_s)ds$$ 
	is $\Pp_x$-martingale for any $x\in\Rd$. 
	Since $|M^g_t|\leq 2||g||_\infty+t||\cL g||_\infty$,  \eqref{e:SnS} follows by the optional stopping theorem.
\end{proof}

\medskip
The following proof is almost the same as that of \cite[Lemma 2.3]{GKK20}.
\begin{lemma}\label{st5}
	There exists a constant $c=c(\phi,  \eta,  \kappa_0, \kappa_1, d )>0$ such that, for any $r \in (0,1)$, $Q\in\partial D$, and any stopping time $S$ (with respect to the filtration of $X$), we have
	\begin{equation}\label{eq:st5-1}
		\mathbb{P}_x\left(|X_{S}-Q|\ge r\right) \,\le\, c \,\frac{ \E_x[S]}{\phi(r)}
		\qquad \text{ for }x \in  B(Q, r/2).
	\end{equation}
	In particular, there exists a constant $C_{\ref{st5}}=C_{\ref{st5}}(\phi,  \eta,  \kappa_0, \kappa_1, d )>0$ such that, for any $r \in (0,1)$, $Q\in\partial D$, and any open sets $U$ with $D_Q(r, r)\subset U  \subset  D$, 
	\begin{equation*}
		\mathbb{P}_x\left(X_{\tau_U} \in D\right) \,\le\, C_{\ref{st5}} \,\frac{ \E_x[\tau_U]}{\phi(r)}
		\qquad \text{ for } x \in D\cap B(Q, r/2{\sqrt{1+\Lambda^2}}).
	\end{equation*}
\end{lemma}
\begin{proof}\hspace{-.05in}.
	For any $x \in B(Q, r/2)$, this lemma is clear when $\E_x[S]=\infty$ so we may assume that $\E_x[S]<\infty$.
	Let $g\in C^{\infty}_c(\bR^d)$ be a function $-1\le g\le 0$ satisfying 
	$$g(y)=\begin{cases}
	-1, & \text{if } |y|< 1/2\\
	0, & \text{if }  |y|\ge 1.
	\end{cases}$$
	Then 
	$$\sum^d_{i=1}\left\|\frac{\partial g}{\partial y_i}
	\right\|_\infty+ \sum^d_{i, j=1}\left\|\frac{\partial^2 g}{\partial y_i\partial y_j}
	\right\|_\infty =c_1<\infty .
	$$
	For any $r \in (0,1)$, define $g_{r}(y)=g(\frac{y-Q}{r})$. Then $-1\le g_{ r}\le 0$,
	\begin{equation}\label{e:n1}
		g_{r}(y)=
		\begin{cases}
			-1, & \text{if }\,|y-Q|< r/2\\
			0, & \text{if }\, |y-Q|\ge r, 
		\end{cases}
	\end{equation}
	and 
	\begin{align*}
		\sum^d_{i=1}\left\|\frac{\partial g_{r}}{\partial y_i} \right\|_\infty<\ c_1 r^{-1} \quad \text{and} \quad  \sum^d_{i, j=1}\left\|\frac{\partial^2g_{r}}{\partial y_i\partial y_j} \right\|_\infty < c_1 r^{-2}.
	\end{align*}
	According to \eqref{cLig}, there exists $c_2=c_2(\phi, \eta,\kappa_0,\kappa_1,d)>0$ such that for $0< r<1$,
	\begin{align}\label{e:genb}
		||\cL g_{r}||_\infty\le \sum_{i=1}^d	||\cL^{(i)} g_{r}||_\infty\le \frac{ c_2}{ \phi(r)}.
	\end{align}
	Therefore, by Lemma \ref{Dyk_F} with \eqref{e:n1} and \eqref{e:genb}, we obtain that for any $x\in B(Q, r/2)$ with $\E_x[S]<\infty$,
	\begin{align*}
		\Pp_x(|X_{S}-Q|\ge r)
		&=\E_x\left[1+g_{r} \left(X_{S}\right);|X_{S}-Q|\ge r\right]\\
		&\le \E_x\left[1+g_{ r} \left(X_{S}\right)\right]=-g_r(x)+\E_x\left[g_{r} \left(X_{S}\right)\right]\\
		&= \E_x\left[ \int_0^{S}   {\cL} g_{r}(X_t)dt \right] \le \|{\cL}g_{r}\|_\infty\,  \E_x[S]\le c_2\frac{\E_x[S]}{\phi(r)},
	\end{align*}
	which gives the proof of \eqref{eq:st5-1}. For the second assertion, for any
	$r\in (0, 1)$ and $Q=(\wt Q, Q^d)\in \partial D$, let $U$ be an open set satisfying that $D_Q(r, r)\subset U  \subset  D$.
	Since $\rho_Q(z)\ge r$ or $|\wt z-\wt Q|\ge r$ for $z=(\wt z, z^d)\in D\backslash U$, using \eqref{comp:rho}, we obtain that
	$$|z-Q|\ge \delta_D(z)\vee |\wt z-\wt Q|\ge \frac{r}{\sqrt{1+\Lambda^2}}.$$
	Therefore, \eqref{eq:st5-1} and ({\bf WS}) imply that for $x \in D\cap B(Q, r/2{\sqrt{1+\Lambda^2}})$,
	\begin{align*}
		\Pp_x\left(X_{\tau_U} \in D\right)\,\le\,    \Pp_x\left(|X_{\tau_U}-Q|\ge \frac{r}{\sqrt{1+\Lambda^2}}\right) \le c \frac{ \E_x[\tau_U]}{\phi(r)}.
	\end{align*}
\end{proof}

\begin{lemma}\label{g:st1-1}
	For any $R_0\le R/2$, 
	let  $r_0={R_0}/{4\Lambda}$.
	For any $Q\in \partial D$ and $k\in\bN$ satisfying that $2^{-k}<r_0/2$, 
	let $$ D_k:=\left\{y\in D\cap B(Q, R): |\wt y-\wt Q|\le r_0-\tfrac{1}{2^k},\ \tfrac{1}{2^k}<\rho_Q(y)<r_0-\tfrac{1}{2^k}\right\}.$$ 
	Then for every $u\in \Rd$ with $|u|<2^{-k}$, recall that the function $h_{R_0}=h_{Q,R_0}$ is defined in \autoref{g:st1}, then
	\begin{align}\label{eq:int1}
		\cL^{(i)}_u h_{Q,R_0}(w):= \lim_{\eps\downarrow 0} \, \int_{\{|t|>\eps\}} (h_{Q,R_0}(w-u+e_it)-h_{Q,R_0}(w-u))J(w, w+e_it)dt
	\end{align}
	is well defined in $D_k$
	and there exists $C_{\ref{g:st1-1}}=C_{\ref{g:st1-1}}(\phi, \kappa_0, \kappa_1, R, \Lambda, \eta, d)>0$ independent of $Q\in  \partial D$ and $k\in\bN$ such that
	\begin{align*}
		| \cL^{(i)}_u h_{Q,R_0}(w)|\le \frac{C_{\ref{g:st1-1}}}{V(R_0)}\qquad \mbox{for all } w\in D_k \text{ and } |u|<2^{-k}.
	\end{align*}
\end{lemma}
\begin{proof}\hspace{-.05in}.
	For any $Q \in \partial D$ and $R_0<R/2$, let $h(y):=h_{Q,R_0}(y)$ and $r_0={R_0}/{4\Lambda}$.  
	For any $w\in D_k$ and  $|u|<2^{-k}$,
	let $x:=w-u$ and $\kappa_u(x, y):=\kappa(u+x, u+y)$.
	Then $x\in D_Q(r_0, r_0)$, 
	and the integral of the right hand side of \eqref{eq:int1} is decomposed by
	\begin{align*}
		G_\eps(x)+\kappa_u(x, x)\cdot \cL^{(i), \eps}_{Z} h(x) +H(x)
	\end{align*}
	where 
	\begin{align*}
		G_\eps(x)&:=\int_{\{\eps<|t|\le 1\}} \big(h(x+e_it)-h(x)\big)\big(\kappa_u(x, x+e_it)-\kappa_u(x, x)\big)\nu^1(|t|)dt,\\
		H(x)&:=\int_{\{1<|t|\}} \big(h(x+e_it)-h(x)\big)\big(\kappa_u(x, x+e_it)-\kappa_u(x, x)\big)\nu^1(|t|)dt,
	\end{align*}
	and $\cL^{(i), \eps}_{Z}$ is defined in \eqref{e:gzi}.
	
	Since $\delta_D(y)\le |y-Q|\le R_0$  for any $y\in B(Q, R_0)$, 
	$\left|h(x+e_it)-h(x)\right|\le c_1$  by the definition of $h$.
	Therefore by \eqref{a:kappa} and  \eqref{e:large} with the fact that $R_0<1$,
	we first obtain 
	\begin{align}\label{e:H}
		|H(x)|\le c_1\int_{\{1<|t|\}}\nu^1(|t|)dt
		\le c_2 V(R_0)^{-1}.
	\end{align}
	
	For any $x, y\in D$, let $z_x,z_y\in\partial D$ be the points satisfying that $|z_x-x|=\delta_D(x)$ and $|z_y-y|=\delta_D(y)$.
	Let $x_0\in D$ be a point constructing a parallelogram with $z_x, z_y, x$ so that $|x-x_0|=|z_x-z_y|$. 
	Since 
	$$|z_x-z_y|^2=|\wt x-\wt y|^2+|\varphi_Q(\wt x)-\varphi_Q(\wt y)|^2\le(1+\Lambda^2) |x-y|^2,$$
	we have that
	$$\left|\delta_D(y)-\delta_D(x)\right|\le
	\left|x_0-y\right|\le |x-y|+|x-x_0|\le 3\Lambda |x-y|.$$
	Therefore, by \eqref{subad} 
	$$\left|h(x+e_it)-h(x)\right|\le c_3 V(|t|).$$
	Also $|\kappa_u(x, x+e_it)-\kappa_u(x, x)|\le \left(\kappa_1 |t|^\eta\wedge 2\kappa_0\right)$ by \eqref{a:kappa} and $\bf(K_\eta)$. Therefore, 
	$$|G_{\eps}(x)|
	\le  c_3\left(\kappa_1 \int_{\{|t|\le R_0\}}V(|t|)|t|^\eta \nu^1(|t|) dt+2\kappa_0\int_{\{|t|> R_0\}}V(|t|)\nu^1(|t|)dt\right)=: c_4\left( \mbox{I}+  \mbox{II}\right).$$
	By \eqref{e:com} and \eqref{V:sc} with $|t|\le R_0$, 
	for $\eta> \ua/2$, 
	$$|\mbox{I}| \le c_5 \int_{\{|t|\le R_0\}}V(|t|)^{-1}|t|^{\eta-1}dt
	\le \frac{c_6 R_0^{\ua/2}}{V(R_0)} \int_{\{|t| \le R_0\}}|t|^{-\ua/2+\eta-1}dt\le\frac{c_7}{V(R_0)}$$
	and by \eqref{eq:Vnu},  
	$|\mbox{II}|\le\frac{c_8}{V(R_0)}$.
	Therefore
	\begin{align}\label{eq:III}
		\lim_{\eps\downarrow 0} G_\eps(x) \text{ exists } \quad \text{ and } \quad  |\lim_{\eps\downarrow 0} G_\eps(x)| \le \frac{c_7+c_8}{V(R_0)}.
	\end{align}
	Finally  from \autoref{g:st1}, 
	$\lim_{\eps\downarrow 0}\cL^{(i), \eps}_{Z} h(x) $ exists and 
	\begin{align}\label{eq:A}
		\left|\lim_{\eps\downarrow 0}\cL^{(i), \eps}_{Z} h(x)  \right|\le \frac{C_{\ref{g:st1}}}{V(R_0)}.
	\end{align}
	Hence combining \eqref{e:H}, \eqref{eq:III} and \eqref{eq:A}, we have the conclusion.
\end{proof}

\medskip
We note the following L\'evy system for $X$: 
for any $x\in \bR^d$, stopping time $S$ (with respect to the filtration of $X$), and non-negative measurable function $f$ on $\bR_+ \times \bR^d\times \bR^d$ with $f(s, y, y)=0$ for all $y\in\bR^d$ and $s\ge 0$, we have that
\begin{equation}\label{eq:LSd}
	\E_x\Big[\sum_{s\le S} f(s,X_{s-}, X_s) \Big]=\E_x \Big[ \int_0^S \Big(\sum_{i=1}^{d}\int_\bR 
	f(s,X_s, X_s+e_ih) 
	J(X_s,X_s+e_ih) dh \Big) ds\Big],
\end{equation}
(see, e.g., \cite[Lemma 4.7]{ChKu03} and \cite[Appendix A]{ChKu08} for the proof). We now arrive at the estimates on exit distributions for $X$.

\begin{theorem}\label{L:exit}
	{\rm(1)}	There exist constants $\lambda_0=\lambda_0(\phi, \kappa_0, \kappa_1, R, \Lambda,  \eta, d )\ge 1$ and $C_{\ref{L:exit}.1}=C_{\ref{L:exit}.1}(\phi, \kappa_0, \Lambda, d )>0$ such that for any $Q\in \partial D$, $s\le R/(2\lambda_0 )$
	and $x\in D_Q(s, s)$,
	\begin{align*}
		\E_x[\tau_{D_Q(s, s)}]&\le C_{\ref{L:exit}.1}V(s)V(\delta_D(x)).
	\end{align*}
	{\rm(2) }
	There exists a constant
	$C_{\ref{L:exit}.2}=C_{\ref{L:exit}.2}(\phi, \kappa_0,\Lambda, d )>0$ such that for any $Q\in \partial D$, $s\le R/2$, $\lambda\ge 4\Lambda$ and $x\in D_Q(\lambda^{-1}s, \lambda^{-1}s)$,
	\begin{align*}
		\Pp_x\Big(X_{\tau_{D_Q(\lambda^{-1}s, \lambda^{-1}s)}}\in  D_Q(s, s)\backslash D_Q(s, \lambda^{-1}s)\Big)&\ge C_{\ref{L:exit}.2} V(\delta_D(x))/V(s).
	\end{align*}
\end{theorem}
\begin{proof}\hspace{-.05in}.
	For any $Q\in \partial D$ and $R_0\le R/2$, define 
	$h_{R_0}(y):=h_{Q, R_0}(y)=V(\delta_D(y))\1_{D\cap B(Q, R_0)} (y)$. 
	Consider a non-negative function $f\in C^\infty(\Rd)$ such that 
	$f(y)=0 \text { for } |y|>1 \text{ and }\int_{\Rd} f(y)dy =1$.
	For any $k\ge 1$, define $f_k(y):=2^{kd} f(2^k y)$ and $h^{(k)}:=f_k*h_{R_0}\in C_c^2(\Rd)$.
	Let 
	$$ D_k^\lambda:=\left\{y\in D\cap B(Q, R):|\wt y-\wt Q|\le \tfrac{R_0}{\lambda}-\tfrac{1}{2^k}, \ \tfrac{1}{2^k}<
	\rho_Q(y)<\tfrac{R_0}{\lambda}-\tfrac{1}{2^k}\right\} \ \text{ for } \lambda\ge 4\Lambda.$$
	Since  $h^{(k)}\in C^2_c(\Rd)$, $\cL^{(i)} h^{(k)}$ is well defined (see, the discussion before \eqref{eq:cLi}) and
	for any $w\in D_k^\lambda$,
	\begin{align}\label{eq:Lik}
		&\int_{\{|t|>\eps\}} ( h^{(k)}(w+e_it)- h^{(k)}(w))J(w, w+e_it)dt\\
		=&\int_{\{|t|>\eps\}}\int_{\bR^d}f_k(u)\left(h_{R_0}(w+e_it-u)-h_{R_0}(w-u)\right)duJ(w, w+e_it)dt\nn\\
		=&\int_{\{|u|<2^{-k}\}}f_k(u)\left(\int_{\{|t|>\eps\}}\left(h_{R_0}(w+e_it-u)-h_{R_0}(w-u)\right)J(w, w+e_it)dt\right)du.\nn
	\end{align}
	\autoref{g:st1-1} implies that for $w\in D_k^\lambda$ and $u\in B(0, 2^{-k})$, 
	\begin{align*}
		&\lim_{\eps\downarrow 0}\int_{\{|t|>\eps\}} (h_{R_0}(w+e_it-u)-h_{R_0}(w-u))J(w, w+e_it)dt= \cL^{(i)}_u h_{R_0}(w),
	\end{align*}
	and $ -C_{\ref{g:st1-1}}V(R_0)^{-1} \le\cL^{(i)}_u h_{R_0}(w) \le C_{\ref{g:st1-1}}V(R_0)^{-1}$.
	So by letting $\varepsilon \downarrow 0$ in \eqref{eq:Lik} with the dominated convergence theorem, it follows that for $w \in D_k^\lambda$,
	\begin{align*}
		|\cL^{(i)} h^{(k)}(w)|=  \left|\int_{\{|u|<2^{-k}\}} f_k(u) \cL^{(i)}_u h_{R_0}(w)\, du\right|\le \frac{C_{\ref{g:st1-1}}}{V(R_0)} \int_{\{|u|<2^{-k}\}} f_k(u) \, du =  \frac{C_{\ref{g:st1-1}}}{V(R_0)}.
	\end{align*}
	Therefore,
	\begin{align}\label{e:generhk}
		|\cL h^{(k)}(w)|\le \sum_{i=1}^d|\cL^{(i)} h^{(k)}(w)|\le d \frac{C_{\ref{g:st1-1}}}{V(R_0)}.
	\end{align}
	Applying \autoref{Dyk_F} to $\tau_{D_k^\lambda}$ and $h^{(k)}$ with \eqref{e:generhk}, for any $x \in D_k^\lambda$,   we have 
	\begin{equation*}
		\E_{x}\left[h^{(k)}\big(X_{\tau_{D_k^\lambda}}\big)\right]-\frac{d C_{\ref{g:st1-1}}}{V(R_0)}  \E_x\big[\tau_{D_k^\lambda}\big] \le h^{(k)}(x) 
		\le\E_{x}\left[h^{(k)}\big(X_{\tau_{D_k^\lambda}}\big)\right]+\frac{d C_{\ref{g:st1-1}}}{V(R_0)}  \E_x\big[\tau_{D_k^\lambda}\big].
	\end{equation*}
	Let $D_0:= D_Q(\lambda^{-1}R_0, \lambda^{-1}R_0)$.
	For any $x \in D_0\subset B(Q, R_0)$, 
	letting $k\to \infty$ we have that
	\begin{align}
		&V(\delta_D(x))\ge \E_{x}\left[h_{R_0}\left(X_{ \tau_{D_0}} \right)  \right] - \frac{d C_{\ref{g:st1-1}}}{V(R_0)} \E_x\left[\tau_{D_0}\right]\,,\label{e:ss}\\ &V(\delta_{D}(x))\le   \E_{x}\left[h_{R_0}\left(X_{ \tau_{D_0}} \right)  \right]+\frac{d C_{\ref{g:st1-1}}}{V(R_0)}\E_x\left[\tau_{D_0}\right].\label{e:ss1}
	\end{align}
We may assume that $R<1/4$.
Let $r_1:=R_0/\lambda<1/(32\Lambda)$.
For $z\in D_0$, since 
\begin{align}\label{e:newD0}
|z-Q|^2\le {r_1^2+(\varphi_Q(r_1)+r_1)^2}\le {r_1^2+(\Lambda r_1^2+r_1)^2}<3r_1^2,\qquad
D_0\subset B(Q, \sqrt 3 r_1).	
\end{align}

In $CS_Q$, let $\Pi$ be the hyperplane tangent to $\partial D$ at $Q$. Let 
$\Gamma(\wt y):=\nabla \varphi_Q(\wt Q)\cdot \wt y$ be a function defined on $\bR^{d-1}$ which describes the hyperplane $\Pi$ and let $v:=v_Q:=\nabla \varphi(\wt x)/|\nabla \varphi(\wt x)|$ be a unit normal vector of the hyperplane $\Pi$. Then there exists an index $i_0\in \{1, \ldots, d\}$ such that the angle between $v$ and $e_{i_0}$, denoted by $\theta_1:=\text{arg}(v, e_{i_0})\in [0, \frac{\pi}{4})$, in other words,
$\left\langle v, e_{i_0}  \right\rangle=\cos \theta_1\in (\frac{1}{\sqrt{2}}, 1]$ and  $|\nabla \varphi_Q(\wt Q)|=|\tan\theta_1|< 1$.	
For any $y\in \bR^d$, we use the new notation $\hat y^i:=(y^1, \ldots, y^{i-1}, y^{i+1}, \ldots, y^d)\in \bR^{d-1}$.
Define two sets $D_1, D_2\subset D$ by
\begin{align*}
P_1:=Q+5r_1e_{i_0}\quad &\text{ and }\quad  D_1:=\{y\in \bR^d:|\hat y^{i_0}-\hat P_1^{i_0}|< \sqrt 3 r_1, \  |y^{i_0}-P_1^{i_0}|< r_1\}\,;\\
P_2:=Q+\big(\tfrac{5r_1}{ \cos\theta_1}\big)e_{i_0}\ &\text{ and }\quad   D_2:=\{y\in \bR^d:|\hat y^{i_0}-\hat P_2^{i_0}|< \sqrt 3 r_1, \  |y^{i_0}-P_2^{i_0}|< r_1\}. 
\end{align*}
Then 
$$D_1\subset D^c_0\cap D\qquad \text{ and } \qquad D_2\subset D_Q(R_0, R_0) \backslash\ D_Q(R_0,r_1).$$
More precisely, for $y\in D_1$, since $|y-P_1|\le 2r_1$, 
$|y-Q|\ge |Q-P_1|-|P_1-y|>5r_1-2r_1=3 r_1$ and \eqref{e:newD0} indicates that $D_1\subset D_0^c$.
Let $B(O, R)$ be the interior ball tangent to $Q$ with a center $O$ and radius $R$. 
We may assume that $\Lambda>2$ so that $r_1<R/16$.  Since $|O-Q|=R$ and $\cos \theta_1\in (\frac{1}{\sqrt{2}}, 1]$, 
$$|O-P_1|^2=(|O-Q|\sin \theta_1)^2+(|O-Q|\cos \theta_1-5r_1)^2=R^2-10r_1  R\cos \theta_1+25r_1^2<(R-2r_1)^2.$$
Therefore, for any $y\in D_1$, 
$$|O-y|\le |O-P_1|+|P_1-y|<(R-2r_1)+2r_1=R,\ \text{hence } D_1\subset B(O, R)\subset D.$$
Also for any $y\in D_2$ in $CS_Q$, note that 
$$|\wt P_2|= \sin\theta_1\cdot (\tfrac{5r_1}{ \cos\theta_1})=5r_1\tan\theta_1<5r_1\qquad\text{ and }\qquad |P_2^d|=\cos\theta_1\cdot (\tfrac{5r_1}{ \cos\theta_1})=5r_1.$$
Then $|\wt y|\le |\wt y-\wt P_2|+|\wt P_2|< 2r_1+5r_1={7r_1}<R_0$. Also since 
$3r_1<|P_2^d|-|y^d-P_2^d|\le |y^d|\le |y^d-P_2^d|+|P_2^d|\le 7r_1$ and the fact that $r_1<\frac{1}{32\Lambda}$,
$$ r_1< 3r_1-\Lambda(7 r_1)^2\le |\rho_Q(y)|=|y^d|-\varphi_Q(|\wt y|)<R_0$$ therefore,
$D_2\subset D_Q(R_0, R_0) \backslash\ D_Q(R_0,r_1)$.
By \eqref{e:ss} with $D_1$ and \eqref{e:ss1} with $D_2$, we obtain our conclusion (1) and (2), respectively.

Let $z  \in D_0$ and $y:=z+le_{i_0}  \in D_1$ for some $l\ge 0$.
Since $z\in B(Q, \sqrt 3 r_1)$ by \eqref{e:newD0}, 
\begin{align*}
2r_1& <|Q^{i_0}-P_1^{i_0}|-|y^{i_0}-P_1^{i_0}|-|z^{i_0}-Q^{i_0}|\le |y^{i_0}-z^{i_0}|=l	\\
&\le |Q^{i_0}-P_1^{i_0}|+|y^{i_0}-P_1^{i_0}|+|z^{i_0}-Q^{i_0}|<5r_1+r_1+\sqrt 3 r_1<8 r_1<\tfrac{1}{4\Lambda}.
\end{align*}
Note that $\rho_Q(y)\ge y^d-\Lambda|\wt y|^2$, and with $\theta_1$ such that $\sin\theta_1<\frac{1}{\sqrt{2}}< \cos\theta_1$,
$$y^d=z^d+ l\cos\theta_1\ge \tfrac{l}{\sqrt{2}}\,, \qquad 
\Lambda |\wt y|^2=\Lambda (|\wt z|+l\sin\theta_1)^2\le \Lambda(\tfrac{1}{2}l+\tfrac{1}{\sqrt{2}}l)^2<\tfrac{9}{16}l.$$ 
So by \eqref{comp:rho}
$$\delta_D(y)\sqrt{1+\Lambda^2}\ge \rho_Q(y)>l/16.$$
	Therefore,  using \eqref{a:kappa} and  \eqref{a:J}
with the above observations, \eqref{eq:LSd}
implies that 
\begin{align}\label{e:31} 
	\E_x\Big[h_{R_0}\big(X_{\tau_{D_0}}\big)\Big]
	&\ge  \kappa_0^{-1}\E_x \Bigg[\int_{D_1} \int_0^{\tau_{D_0}} h_{R_0}(y) \nu^1( |X_t^{i_0}-y^{i_0}|)\, dt  dy \Bigg]\\
	&\ge  \kappa_0^{-1} \,  \E_x\left[\tau_{D_0}\right] \int_{2r_1}^{8r_1} V(l/16\sqrt{(1+\Lambda^2)}) \nu^1(|l|)dl\nn \\
	&\ge c_1\,  \E_x\left[\tau_{D_0}\right] \int_{2r_1}^{8r_1} \frac{1}{V(l)l}dl\ge c_1\E_x\left[\tau_{D_0}\right] C_{\ref{VbisEst}}^{-1}\int_{2r_1}^{8r_1}\frac{V^{\prime}(l)}{V(l)^2}dl\nn\\
	&=  C_2\E_x\left[\tau_{D_0}\right] V(\lambda^{-1}R_0)^{-1}\nn 
\end{align}
for some $C_2:=C_2(\phi, \kappa_0, d)>0$.
Here we also used the fact that $V$ is non-decreasing, \eqref{e:com} and  \autoref{VbisEst}.
	Let $\lambda_0:=\Big(2C_VdC_{\ref{g:st1-1}}/C_2\Big)^{2/\la}\vee 4\Lambda\ge 1$.
Using \eqref{V:sc},  $V(\lambda^{-1}R_0)\le V(\lambda_0^{-1}R_0)\le C_V\lambda_0^{-\la/2}V(R_0)$ for any $\lambda\ge \lambda_0$.
Then combining \eqref{e:ss} and  \eqref{e:31},  we have that for $\lambda \ge \lambda_0$,  
\begin{align*}
	V(\delta_{D}(x))
	\ge & \left(C_2V( \lambda^{-1}R_0)^{-1} -dC_{\ref{g:st1-1}}V(R_0)^{-1} \right) \E_x[\tau_{D_0}]\\
	\ge & \, (C_2/2)\, V(\lambda^{-1}R_0)^{-1}\E_x[\tau_{D_0}].
\end{align*}
Thus, we have proved  (1) with $\lambda_0$ and $s=\lambda^{-1}R_0$ where $\lambda\ge \lambda_0$.

For the second assertion, according to the definition of $h_{R_0}$, 
we first note that $$\E_{x}\big[h_{R_0}\big(X_{\tau_{D_0}}\big)\big]\le c_3 V(R_0)\Pp_{x}\big(X_{ \tau_{D_0}}\in D\big).$$ 
Combining this fact and  \autoref{st5} with $D_0=D_Q(\lambda^{-1}R_0,\lambda^{-1}R_0)$,
\eqref{e:ss1} together with  \eqref{e:com} yields that for $x \in D_0$,
\begin{align}\label{e:39}
V(\delta_D(x))\le &\,\, c_4\left( V(R_0)\,\Pp_{x}\big(X_{\tau_{D_0}}\in D \big)+  V(R_0)^{-1}\E_x\left[\tau_{D_0}\right]\right)\nn\\
\le &\,\, c_4\left( C_{\ref{st5}}V(R_0)\,\phi(\lambda^{-1}R_0)^{-1}+  V(R_0)^{-1}\right)\E_x\left[\tau_{D_0}\right]\nn\\
\le&\ c_5 V(R_0)\left(V(\lambda^{-1}R_0)^{-2}+V(R_0)^{-2}\right)\E_x\left[\tau_{D_0}\right.
\end{align}
Recall that $$P_2=Q+\big(\tfrac{5r_1}{ \cos\theta_1}\big)e_{i_0}\ \text{ and }\  D_2=\{y\in \bR^d:| \hat y^{i_0}-\hat P_2^{i_0}|< \sqrt 3 r_1, \  |y^{i_0}-P_1^{i_0}|< r_1\}. $$	
For $z\in D_0$ and $y=z+le_{i_0}\in D_2$ with some $l>0$,
\begin{align*}
	r_1<&|P_2^{i_0}-Q^{i_0}|-|Q^{i_0}-z^{i_0}|-|y^{i_0}-P_2^{i_0}|\le	|y^{i_0}-z^{i_0}|=l\\
\le &|P_2^{i_0}-Q^{i_0}|+|Q^{i_0}-z^{i_0}|+|y^{i_0}-P_2^{i_0}|\le \tfrac{5r_1}{\cos\theta_1}+\sqrt 3 r_1+r_1\le 10r_1.
\end{align*}
Therefore, by \eqref{a:kappa}, \eqref{eq:LSd}, \eqref{e:com} and \eqref{V:sc}, we obtain that 
\begin{align}\label{e:40}
\Pp_{x}\Big(X_{\tau_{D_0}}&\in  \big(D_Q(R_0, R_0) \backslash\ D_Q(R_0, r_1)\big)\Big)\ge \Pp_{x}\Big(X_{\tau_{D_0}}\in D_2\Big)\nn\\
&\ge \,\kappa_0^{-1}\E_x\left[\int_{D_2}\int_0^{\tau_{D_0}}\nu^1( |X_t^{i_0}-y^{i_0}|)dt\, dy \right]
\ge \ c_6\ \E\left[\tau_{D_0}\right] \int_{r_1}^{2r_1}\nu^1(|l|)dl\nn\\
&\ge c_7\ \E\left[\tau_{D_0}\right] \int_{r_1}^{2r_1}\frac{1}{[V(l)]^2  l}\ dl
\ge  c_8\ V(r_1)^{-2}\E\left[\tau_{D_0}\right].
\end{align}

 Recall that $r_1=\lambda^{-1}R_0$,	combining \eqref{e:39}, \eqref{e:40}, and the fact that $V$ is non-decreasing, we conclude that for $\lambda\ge 4\Lambda$ and for $x \in B(Q, R_0/(2\lambda{\sqrt{1+\Lambda^2}}))$,
\begin{align*}
V(\delta_D(x))\le &\ c_5c_8	V(R_0)\left(V(\lambda^{-1}R_0)^{-2}+V(R_0)^{-2}\right) V(\lambda^{-1}R_0)^2\\
&\qquad\cdot \Pp_{x}\left(X_{ \tau_{D_0}}\in \big(D_Q(R_0, R_0) \backslash\ D_Q(R_0, \lambda^{-1}R_0)\big)\right)\nn\\
\le&\ c_9V(R_0)\Pp_{x}\left(X_{ \tau_{D_0}}\in \big(D_Q(R_0, R_0) \backslash\ D_Q(R_0, \lambda^{-1}R_0)\big)\right).
\end{align*}
Thus, we have proved (2) with $s=R_0$.
\end{proof}

\section{Regularity of Dirichlet heat kernel and upper bound estimate}\label{sec:regupper}
{Let $U\subset\bR^d$ be an open set.
	For any $t>0, x,y\in\bR^d$,  we consider the function
	$$p_U(t,x,y):=p(t,x,y)-\E_x[\, p(t-\tau_U,X_{\tau_U},y);\ \tau_U<t].$$
	It can be pointwisely defined and it is jointly measurable on $(0,\infty)\times(\bR^d\setminus \partial U)\times(\bR^d\setminus \partial U)$.
	
	\medskip
	In \autoref{subsec:regularity}, we obtain the regularity result of $p_U(t,x,y)$ in \autoref{thm:regularity}, which states that $p_U(t,x,y)$ admits a continuous refinement. In \autoref{subsec:upper}, we obtain the short time sharp upper bound for the Dirichlet heat kernel, which is the result in \autoref{t:nmain}.}

\subsection{Regularity of Dirichlet heat kernel}\label{subsec:regularity}
Using the same methods as in the proof of \cite[Theorem 2.4]{MR1329992}, one obtains the following fundamental properties for $p_U(t,x,y)$.
\begin{lemma}\label{CZlemma}
	For any $t>0, A\in\mathcal B(\bR^d)$ and $x\in \bR^d$,
	$$
	\Pp_x(t<\tau_U;X_t\in A)=\int_Ap_U(t,x,y)dy.
	$$
	$p_U(t,x,y)$ is almost surely symmetric on $\bR^d\times\bR^d$, that is, for any $t>0$,
	$$
	p_U(t,x,y)=p_U(t,y,x)\qquad \text{ for a.e. } x,y\in U\times U.
	$$
	Moreover, we have for any $s,t>0$ and $x\in U$,
	$$
	p_U(t+s,x,y)=\int_{\bR^d}p_U(t,x,z)p_U(s,z,y)dz\qquad\text{for a.e. }y\in U.
	$$
\end{lemma}
The above properties determine a semigroup, that is, for any non-negative measurable function $f$ on $U$ or $f\in L^2(U)\cap L^\infty(U)$, 
$$P_t^Uf(x):=\int_Up_U(t,x,z)f(z)dz\qquad \text{for }t>0,\ x\in\bR^d\setminus\partial U,$$
and the heat semigroup is pointwisely defined on $(0,\infty)\times(\bR^d\setminus \partial U)$. Moreover,  for any $\lambda>0$, we define the $\lambda$-potential as follows:
$$U^\lambda f(x)=\E_x\left[\int_0^\infty e^{-\lambda t}f(X_t^U)dt\right]=\int_0^\infty e^{-\lambda t}P_t^Uf(x)dt\qquad\text{for } x\in\bR^d\setminus\partial U,$$
and it is also defined pointwisely on $\bR^d\setminus\partial U$. The following results are established in \cite[Proposition 4.2]{KW22}.
\begin{proposition}\label{prop:Exptau}
	For any $r>0$, there exist $a_i:=a_i(\phi, d)>0, i=1,2$ such that
	$$a_1\phi(r)\leq \E_x[\tau_{B(x,r)}]\leq a_2\phi(r).$$
\end{proposition}
\medskip

A function $h:\bR^d \to \bR$ is said to be {\it harmonic} in the open set $D$ with respect to $X$  if for every open set $U\subset D$ whose closure is a compact subset of $D$, 
$\E_x[|h(X_{\tau_U})|]<\infty$ and 
\begin{equation}\label{eq:harm}
	h(x)=\E_x[h(X_{\tau_U})]\qquad \mbox{ for every } x\in U.
\end{equation}
It is  said  that $h$ is {\it regular harmonic} in $D$ with respect to $X$ if $h$ is harmonic in $D$ with respect to $X$ and \eqref{eq:harm} holds for $U=D$.

\medskip

We now have the H\"older continuity of harmonic functions according to \cite[Theorem 4.6]{KW22}. 
For any $x_0\in\bR^d$ and $r\in(0,1]$, suppose that $h$ is harmonic in $B(x_0,r)$ with respect to $X$ and bounded in $\bR^d$. Then there exist constants $a_3,\beta>0$ depending on $\phi$ and $d$ such that
\begin{equation*}
	\label{h:Holder}
	|h(x)-h(y)|\leq a_3\left(\frac{|x-y|}{r}\right)^\beta\|h\|_\infty\qquad
	\text{ for any }x,y\in B(x_0,r/2).
\end{equation*}
\begin{proposition}\label{prop:Ulambda}
	Let $r\in(0,1]$ and $B(x_0,r)\subset U$, then there exists a constant $c=c(\phi, \lambda, d)>0$ such that for any $x,y\in B(x_0,r/2)$ and $f\in L^\infty(U)\cap L^2(U)$,
	$$
	|U^\lambda f(x)-U^\lambda f(y)|\leq c\left(\phi(r)+\frac{|x-y|^\beta}{r^\beta}\right)\|f\|_{L^\infty(U)}.
	$$
\end{proposition}
\begin{proof}\hspace{-.05in}.
	For simplicity of notations, we write $\tau_x:=\tau_{B(x,r)}$. For any $f\in L^\infty(U)\cap L^2(U)$, by the  strong Markov property, we have
	$$U^\lambda f(x)=\E_x\left[\int_0^{\tau_x}e^{-\lambda t}f(X_t^U)dt\right]+\E_x\left[(e^{-\lambda\tau_x}-1)U^\lambda f(X_{\tau_x}^U)\right]+\E_x\left[U^{\lambda}f(X_{\tau_x}^U)\right].$$
	We get similar expressions when $x$ is replaced by $y$. By \autoref{prop:Exptau}, the mean value theorem,  and the fact that $\|U^\lambda f\|_\infty\leq \frac{1}{\lambda}\|f\|_\infty$, we obtain that
	\begin{align*}
		|U^\lambda f(x)-U^\lambda f(y)|\leq&\ (\E_x[\tau_x]
		+\E_y[\tau_y])(\|f\|_\infty+\lambda\|U^\lambda f\|_\infty) +\Big|\E_x[U^{\lambda}f(X_{\tau_x}^U)]-\E_y[U^{\lambda}f(X_{\tau_y}^U)]\Big|\\
		\leq&\ 4 a_2
		\phi(r)\|f\|_\infty+|\E_x[U^{\lambda}f(X_{\tau_x}^U)]-\E_y[U^{\lambda}f(X_{\tau_y}^U)]|.
	\end{align*}
	Since $z\mapsto \E_z[U^{\lambda}f(X_{\tau_z}^U)]$ is bounded and harmonic in $B(x_0,r)$, by the H\"older continuity of harmonic functions in \eqref{h:Holder}, for any $x,y\in B(x_0,r/2)$,
	$$\Big|\E_x[U^{\lambda}f(X_{\tau_x}^U)]-\E_y[U^{\lambda}f(X_{\tau_y}^U)]\Big|\leq a_3\frac{|x-y|^\beta}{r^\beta}\|U^\lambda f\|_\infty.$$
	Combining the above estimates, we obtain that for any $x,y\in B(x_0,r/2)$,
	$$|U^\lambda f(x)-U^\lambda f(y)|\leq c\left(\phi(r)+\frac{|x-y|^\beta}{r^\beta}\right)\|f\|_{L^\infty(U)}$$
	where $c$ is a constant depending on $\phi, d,\lambda$.
\end{proof}

\medskip
Since $\{P_t^U,t>0\}$ is a heat semigroup on $L^2(U)$, according to spectral theory, there exists a spectral family $\{E_\mu,\mu\in\bR\}$ such that
$$f=\int_0^\infty dE_\mu(f),\ P_t^Uf=\int_0^\infty e^{-\mu t}dE_\mu (f)\ \text{ and }\ U^\lambda f=\int_0^\infty \frac{1}{\lambda+\mu}dE_\mu(f).$$
For any $f\in L^\infty(U)\cap L^2(U)$, set 
$$h:=h(f):=\int_0^\infty (\lambda+\mu)e^{-\mu t}dE_\mu(f).$$
For any $g\in L^1(\bR^d)$, by \eqref{eq:hke}, we have that for any $t>0$,
$$\|P_t^Ug\|_\infty=\bigg\|\int_Up_U(t,\cdot,z)g(z)dz\bigg\|_\infty\leq\bigg\|\int_U|p(t,\cdot,z)||g(z)|dz\bigg\|_\infty\leq c[\phi^{-1}(t)]^{-d}\|g\|_1.$$
Correspondingly,
$$\|P_t^Ug\|_2\leq\sqrt{\|P_t^Ug\|_1\|P_t^Ug\|_\infty}\leq\sqrt{\|g\|_1c[\phi^{-1}(t)]^{-d}\|g\|_1}\leq \sqrt{c}[\phi^{-1}(t)]^{-d/2}\|g\|_1.$$
Thus, we conclude that $P_t^Ug\in L^\infty(\bR^d)\cap L^2 (\bR^d)$. 
Then by Cauchy-Schwartz inequality  with the fact that
$\sup_{\{\mu>0\}}(\lambda+\mu)e^{-\mu t/2}\le 2(t^{-1}\vee\lambda)$,
\begin{align*}
	\langle h,g\rangle=&\int_0^\infty (\lambda +\mu)e^{-\mu t}d\langle E_\mu(f),E_\mu (g)\rangle\\
	\leq &\left(\int_0^\infty (\lambda+\mu)e^{-\mu t}d\langle E_\mu(f),E_\mu (f)\rangle\right)^{1/2}\left(\int_0^\infty (\lambda+\mu)e^{-\mu t}d\langle E_\mu(g),E_\mu (g)\rangle\right)^{1/2}\\
	\leq&\ 2 (t^{-1}\vee\lambda)\left(\int_0^\infty d\langle E_\mu(f),E_\mu (f)\rangle\right)^{1/2}\left(\int_0^\infty e^{-\mu t/2}d\langle E_\mu(g),E_\mu (g)\rangle\right)^{1/2}\\
	=&\ 2(t^{-1}\vee\lambda)\|f\|_{L^2(U)}\|P_{t/4}^Ug\|_2\leq c (t^{-1}\vee\lambda)[\phi^{-1}(t)]^{-d/2}\|f\|_{L^2(U)}\|g\|_1.
\end{align*}
Since we take $g\in L^1(\bR^d)$ {to be arbitrary}, we get $$\|h\|_\infty\leq c (t^{-1}\vee\lambda)[\phi^{-1}(t)]^{-d/2}\|f\|_{L^2(U)}.$$

\begin{theorem}\label{thm:regularity}
	For $f\in L^2(U) \cap L^\infty (U)$, $P_t^Uf$ is equal a.e. to a function that is (locally) H\"older continuous. Hence, we can refine $p_U(t,x,y)$ to be jointly (locally) H\"older continuous for any $t>0$ and $x,y\in U$.
\end{theorem}
\begin{proof}\hspace{-.05in}.
	The above discussion and \autoref{prop:Ulambda} yield that, for fixed $t>0$, $r\in(0,1]$ and $B(x_0,r)\subset U$, there exists a constant $c_1$ such that for any $x,y\in B(x_0,r/2)$,
	\begin{align}\label{Ptdiff}
		|P_t^Uf(x)-P_t^Uf(y)|=&|U^\lambda h(x)-U^\lambda h(y)|\nn\\
		\leq &\ c_1\left(\phi(r)+\frac{|x-y|^\beta}{r^\beta}\right)\|h\|_{L^\infty(U)}\nn\\
		\leq&\ c_1 (t^{-1}\vee\lambda)[\phi^{-1}(t)]^{-d/2}\left(\phi(r)+\frac{|x-y|^\beta}{r^\beta}\right)\|f\|_{L^2(U)}.
	\end{align}
	For any fixed compact set $K\subset U$ and {any points} $x,y \in K$, let $x_0=x$ and $\delta_K=\frac{1}{4}(\text{dist}(K,\partial U)\wedge\text{dist}(K,\partial U)^2\wedge 1)$.
	If $|x-y|<\delta_K$, then with $r:=|x-y|^{1/2}$, we have 
	$$\frac{1}{2}\text{dist}(K,\partial U)\geq \sqrt{\delta_K}>|x-y|^{1/2}=r\geq2|x-y|.$$
	Then applying \eqref{Ptdiff} with $x_0=x$ and $r=|x-y|^{1/2}$, and  ({\bf WS}) with $r<1$, 
	$$|P_t^Uf(x)-P_t^Uf(y)|\leq  c_2 (t^{-1}\vee\lambda)[\phi^{-1}(t)]^{-d/2}\left(|x-y|^{\la/2}+|x-y|^{\beta/2}\right)\|f\|_{L^2(U)}.$$
	If $|x-y|\geq\delta_K$, by the definition of $P_t^Uf$ and \eqref{eq:hke}, we have that for any $t>0$ and $z\in U$,
	\begin{align*}
		|P_t^Uf(z)|=&\left|\int_Up_U(t,z,y)f(y)dy\right|\leq\sqrt{\int_Up_U(t,z,y)^2dy}\|f\|_{L^2(U)}\\
		\leq&\sqrt{p(2t,z,z)}\|f\|_{L^2(U)}\leq \sqrt{c_3}[\phi^{-1}(2t)]^{-d/2}\|f\|_{L^2(U)}.
	\end{align*}
	Hence, for $x,y\in K$ with $|x-y|\geq\delta_K$, 
	\begin{align*}
		|P_t^Uf(x)-P_t^Uf(y)|\leq&|P_t^Uf(x)|+|P_t^Uf(y)|\leq 2\sqrt{c_3}[\phi^{-1}(2t)]^{-d/2}\|f\|_{L^2(U)}\\
		\leq&\sqrt{c_3}[\phi^{-1}(2t)]^{-d/2}\left(\frac{|x-y|^{\la/2}}{\delta_K^{\la/2}}+\frac{|x-y|^{\beta/2}}{\delta_K^{\beta/2}}\right)\|f\|_{L^2(U)}.
	\end{align*}
	We have our first assertion by combining the above two discussions and ({\bf WS}), that is, for any compact set $K\subset U$ and $x,y\in K$,
	\begin{equation}
		\label{eq:PtHolder}
		|P_t^Uf(x)-P_t^Uf(y)|\leq c_4(\phi, d,\text{dist}(K,\partial U))|x-y|^{(\la\wedge\beta)/2}[\phi^{-1}(t)]^{-d/2}\|f\|_{L^2(U)}.
	\end{equation}
	Let ${\wt f}(z)=p_U(t/2,z,y)$. According to \autoref{CZlemma}, since
	$$p_U(t,x,y)=\int_Up_U(t/2,x,z)p_U(t/2,z,y)dz=P_{t/2}^U{\wt f}(x),$$
	by \eqref{eq:PtHolder} and \eqref{eq:hke},
	we have that for any compact set $K\subset U$ and $x,x',y\in K$,
	\begin{align*}
		|p_U(t,x,y)-p_U(t,x',y)|=&|P_{t/2}^U{\wt f}(x)-P_{t/2}^U{\wt f}(x')|\\
		\leq&c_4|x-x'|^{(\la\wedge\beta)/2}[\phi^{-1}(t)]^{-d/2}\|p_U(t/2,\cdot,y)\|_{L^2(U)}\\
		\leq&c_5|x-x'|^{(\la\wedge\beta)/2}[\phi^{-1}(t)]^{-d/2}\sqrt{p(t,y,y)}\\
		\leq&c_6|x-x'|^{(\la\wedge\beta)/2}[\phi^{-1}(t)]^{-d}.
	\end{align*}
	With the symmetry property in \autoref{CZlemma}, $p_U(t,x,y)$ is jointly H\"older continuous with constants independent of $x,y$ on $K\times K$, hence, it is jointly (locally) H\"older continuous on $U\times U$.

\end{proof}

\medskip

\subsection{The upper bound}\label{subsec:upper}
We first obtain a key lemma, \autoref{lemma:pU}, that serves as a guideline to obtain the upper bound for $p_D(t,x,y)$. Together with \autoref{L:exit} and \eqref{eq:hke}, then we derive the upper bound with one boundary decay, see \autoref{prop:upper1}. Finally, we prove \autoref{t:nmain}.

\begin{lemma}\label{lemma:pU}
	Let $U\subset \Rd$ be an open set and $U_1,U_3\subset U$ be disjoint open sets with $\text{dist}(U_1,U_3)>0$.  Define $U_2:=U\setminus(U_1\cup U_3)$. Then for any $t>0,x\in U_1$ and $y\in U_3$,
	\begin{align}\label{eq:pU}
		&p_U(t,x,y)
		\leq \ \Pp_x(X_{\tau_{U_1}}\in U_2)\sup_{t/2<s<t,z\in U_2}p_U(s,z,y)+\frac{2}{t}\ \E_x[\tau_{U_1}]\sup_{z\in U_1}p_U(t/2,z,y)\\
		&+\int_{0}^{t/2}\int_{U_1}p_{U_1}(s,x,u)\Bigg(\sum_{i=1}^d\int\limits_{\{h\in \bR: u+e_ih\in U_3\}}\hspace{-.2in}p_U(t-s,u+e_ih,y)dh\cdot\hspace{-.1in}\sup_{\substack{u\in U_1\\ u+e_ih\in U_3}}\hspace{-.1in}J(u,u+e_ih)\Bigg)duds.\nn
	\end{align}
\end{lemma}
\begin{proof}\hspace{-.05in}.
	Let $x\in U_1$ and  $f$ be a non-negative function in $L^1(U)\cap L^\infty(U)$.
	By the strong Markov property of $X$, we first obtain
	\begin{align}\label{eq:smp1}
		&P_t^Uf(x)=\E_x[f(X_t);t<\tau_U]=\E_x[f(X_t);t<\tau_{U_1}]+\E_x[f(X_t^U);\tau_{U_1}<t]\\
		=&\ P_t^{U_1}f(x)+\E_x[P_{t-\tau_{U_1}}^Uf(X_{\tau_{U_1}});\tau_{U_1}<t,X_{\tau_{U_1}}\in U_2]	 +\E_x[P_{t-\tau_{U_1}}^Uf(X_{\tau_{U_1}});\tau_{U_1}<t,X_{\tau_{U_1}}\in U_3]\nn\\
		\leq&\  P_t^{U_1}f(x)+\sup_{s<t,z\in U_2}P_s^Uf(z)\cdot\Pp_x(X_{\tau_{U_1}}\in U_2)+\E_x[P_{t-\tau_{U_1}}^Uf(X_{\tau_{U_1}});\tau_{U_1}<t,X_{\tau_{U_1}}\in U_3].\nn
	\end{align}
	By  the L\'evy system of $X$ in \eqref{eq:LSd}, we have
	\begin{align}\label{eq:ls2}
		&	\E_x[P_{t-\tau_{U_1}}^Uf(X_{\tau_{U_1}});\tau_{U_1}<t,X_{\tau_{U_1}}\in U_3]\\
		=&\ \int_0^t\int_{U_1}p_{U_1}(s,x,u)\Bigg(\sum_{i=1}^d\int\limits_{\{ h\in \bR:u+e_ih\in U_3\}}\hspace{-.2in}P_{t-s}^Uf(u+e_ih)J(u,u+e_ih)dh\Bigg)duds\nn\\
		\leq&\  \int_0^t\int_{U_1}p_{U_1}(s,x,u)\Bigg(\sum_{i=1}^d\int\limits_{ \{h\in \bR: u+e_ih\in U_3\}}\hspace{-.2in}P_{t-s}^Uf(u+e_ih)dh\cdot\hspace{-.1in} \sup_{\substack{u\in U_1\\ u+e_ih\in U_3}}\hspace{-.1in}J(u,u+e_ih)\Bigg)duds.\nn
	\end{align}
	Letting $f(\cdot)=p_U(t,\cdot,y)$ in \eqref{eq:smp1}  and \eqref{eq:ls2}, the semigroup property implies that for any $y\in U_3$, 
	\begin{align}\label{eq:smp2}
		&	p_U(2t,x,y)\leq \int_{\Rd}p_{U_1}(t,x,z)p_U(t,z,y)dz+\sup_{t<s<2t,z\in U_2}p_U(s,z,y)\cdot\Pp_x(X_{\tau_{U_1}}\in U_2)\\
		&+\int_0^t\int_{U_1}p_{U_1}(s,x,u)\times\Bigg(\sum_{i=1}^d\int\limits_{\{h\in \bR: u+e_ih\in U_3\}}\hspace{-.3in}p_U(2t-s,u+e_ih,y)dh\cdot \hspace{-.1in}\sup_{\substack{u\in U_1\\ u+e_ih\in U_3}}\hspace{-.1in}J(u,u+e_ih)\Bigg)duds.\nn
	\end{align}
	Since the first term in \eqref{eq:smp2} satisfies
	\begin{align*}
		\int_{\Rd}p_{U_1}(t,x,z)p_U(t,z,y)dz \leq  \sup_{z\in U_1}p_U(t,z,y)\cdot\Pp_x(\tau_{U_1}>t)	\leq  \sup_{z\in U_1}p_U(t,z,y)\frac{\E_x[\tau_{U_1}]}{t},
	\end{align*}
	we achieve the result by replacing $2t$ by $t$ in \eqref{eq:smp2}.
\end{proof}

\begin{lemma}\label{lemma:tau}
	For any $T>0$, there exists a constant  $C_{\ref{lemma:tau}}=C_{\ref{lemma:tau}}(\phi,\kappa_0,\kappa_1, R,\Lambda,T, \eta,  d)>0$ such that for any $x\in D$ and $t\leq T$,
	$$\Pp_x\left(\tau_D\geq t\right)\leq C_{\ref{lemma:tau}}\left(1\wedge \frac{V(\delta_D(x))}{\sqrt t}\right).$$
\end{lemma}
\begin{proof}\hspace{-.05in}.
	Let $a_T:=T^{-1}[V(R/(4\lambda_0))]^2$, where $\lambda_0\ge 1$ is the constant from \autoref{L:exit}.
	By the definition of $(R, \Lambda)$, we may assume that $R<1$ and $\Lambda>2$.
	Also it is enough to  consider the case that $V({2\Lambda}\cdot \delta_D(x))<\sqrt{a_T t}$, so that  $\delta_D(x)<V^{-1}(\sqrt{a_Tt})/{ \sqrt{1+\Lambda^2}}\leq R/(4\lambda_0)<1/4$.
	Now let $r_t:=V^{-1}(\sqrt{a_Tt})$ and  $Q\in\partial D$ such that $\delta_D(x)=|x-Q|$. Applying \autoref{st5} with $U:=D_Q(2r_t,2r_t)$, 
	\begin{align*}
		\Pp_x\left(\tau_D\geq t\right)&=\Pp_x\left(\tau_{U}\geq t,\tau_{D}=\tau_{U}\right)+\Pp_x(\tau_D>\tau_{U}, \tau_D\geq t)\\
		&\leq \Pp_x(\tau_{U}\geq t)+\Pp_x(X_{\tau_{U}}\in D)\leq \frac{\E_x[\tau_U]}{t}+C_{\ref{st5}}\frac{\E_x[\tau_U]}{\phi(2r_t)}.
	\end{align*}
	Note that by \eqref{e:com}  and \eqref{V:sc}, $\phi(2r_t)\asymp [V(2r_t)]^2\asymp a_Tt$. Applying \autoref{L:exit}(1) yields that
	\begin{align*}
		\Pp_x(\tau_D\geq t)\leq  c_1(1+C_{\ref{st5}}a_T^{-1})C_{\ref{L:exit}.1}\frac{V(2r_t)V(\delta_D(x))}{t}\leq c_2\frac{V(\delta_D(x))}{\sqrt t}.
	\end{align*}
\end{proof}
\begin{proposition}
	\label{prop:upper1}
	For any $T>0$, there exists a constant $C_{\ref{prop:upper1}}=C_{\ref{prop:upper1}}(\phi,\kappa_0,\kappa_1, R,\Lambda,T,\eta,  d)>0$  such that for any $x,y\in D$ and $t\le T$,
	$$p_D(t,x,y)\leq C_{\ref{prop:upper1}}\Psi(t,x)p(t,x,y),$$
	where $\Psi(t,x)$ is given in \eqref{e:dax}.
\end{proposition}
\begin{proof}\hspace{-.05in}.
	For any fixed $T>0$, let $t\in(0,T]$. We choose and fix a constant $a<a_T :=T^{-1}[V(R/(4\lambda_0))]^2$ such that $r_t:=V^{-1}(\sqrt{at})<1$ for all $t\in[0,T]$. It suffices to show the result for $\delta_D(x)\leq r_t/2{ \sqrt{1+\Lambda^2}}$.
	We deal with two cases separately.\\
	\textit{Case 1.} For all $i\in\{1,\dots,d\}$, let $|x^i-y^i|<4 r_t$. By the semigroup property, \eqref{eq:hke} and \autoref{lemma:tau}, we have
	\begin{align*}
		p_D(t,x,y)&\leq  \sup_{z\in D}p_D(t/2,z,y)\int_Dp_D(t/2,x,z)dz\\
		&\leq c_1[\phi^{-1}(t/2)]^{-d}\Pp_x(\tau_D>t/2)\leq c_2[\phi^{-1}(t/2)]^{-d}\left(1\wedge \frac{V(\delta_D(x))}{\sqrt t}\right).
	\end{align*}
	Since $|x^i-y^i|<4 r_t\asymp \phi^{-1}(t)$, according to \eqref{eq:hke} and  \eqref{e:com}, we have our desired result.\\
	\textit{Case 2.} There exists $i\in\{1,\dots,d\}$ such that $|x^i-y^i|\geq 4 r_t$. Consider the set of indices of the coordinate for $x, y$ that
	$$ E:=\{i\in\{1,\dots,d\}:|x^i-y^i|\geq 4r_t\}.$$
	Let $Q\in\partial D$ such that $\delta_D(x)=|x-Q|$. Define $U_1:=D_Q(r_t,r_t)$ and $U_3:=\{z\in D:\text{ there exists }i\in E \text{ such that }|z^i-x^i|>|x^i-y^i|/2\}$. Then $\text{dist}(U_1, U_3)>0$. Let $U_2:=D\setminus (U_1\cup U_3).$ 
	To obtain our assertion we will estimate the right-hand side of \eqref{eq:pU}.  Note that  $x\in U_1$, $y\in U_3$ and  if $i\notin E$, then
	$$|x^i-y^i|<4V^{-1}(\sqrt{at})\asymp\phi^{-1}(t).$$ For any $z\in U_2$,  if $i\in E$,
	$$|z^i-y^i|\geq |x^i-y^i|-|z^i-x^i|\geq|x^i-y^i|/2.$$
	Hence, by \eqref{eq:hke}, \autoref{st5} and \autoref{L:exit}, we have
	\begin{align}\label{eq:u1}
		&	\Pp_x(X_{\tau_{U_1}}\in U_2)\sup_{t/2<s<t,z\in U_2}p_U(s,z,y)\nn\\
		\leq &\	\Pp_x(X_{\tau_{U_1}}\in D)\sup_{t/2<s<t,z\in U_2}p(s,z,y)\nn\\
		\leq &\ C_{\ref{st5}}\frac{\E_x[\tau_{U_1}]}{\phi(r_t)}\sup_{t/2<s<t,z\in U_2}\left\{c_3[\phi^{-1}(s)]^{-d}\prod_{i\in E}\left(1\wedge\frac{s\phi^{-1}(s)}{|z^i-y^i|\phi(|z^i-y^i|)}\right)\prod_{i\notin E}1\right\}\nn\\
		\leq&\ c_4\frac{V(\delta_D(x))}{\sqrt t}[\phi^{-1}(t)]^{-d}\prod_{i=1}^d\left(1\wedge\frac{t\phi^{-1}(t)}{|x^i-y^i|\phi(|x^i-y^i|)}\right).
	\end{align}
	For the second term of \eqref{eq:pU}, let $z\in U_1$. Then $|z^i-x^i|<2r_t$, and  for  $i\in E$,
	$$|z^i-y^i|\geq|x^i-y^i|-|z^i-x^i|\geq |x^i-y^i|-2r_t\geq |x^i-y^i|/2.$$
	With \autoref{L:exit}(1) and \eqref{eq:hke}, we arrive at
	\begin{align}\label{eq:u2}
		&\frac{2}{t}\ \E_x[\tau_{U_1}]\sup_{z\in U_1}p_U(t/2,z,y)\nn\\
		\leq&\ C_{\ref{L:exit}.1}\frac{2}{t}V(r_t)V(\delta_D(x))\cdot C_1[\phi^{-1}(t/2)]^{-d}\prod_{i\in E}\left(1\wedge \frac{t\phi^{-1}(t/2)}{|x^i-y^i|\phi(|x^i-y^i|/2)}\right)\prod_{i\notin E}1\nn\\
		\leq&\ c_5\frac{V(\delta_D(x))}{\sqrt t}[\phi^{-1}(t)]^{-d}\prod_{i=1}^d\left(1\wedge\frac{t\phi^{-1}(t)}{|x^i-y^i|\phi(|x^i-y^i|)}\right)
	\end{align}
	where the last inequality follows from $|x^i-y^i|<4V^{-1}(\sqrt{at})\asymp\phi^{-1}(t)$ for $i\notin E$.
	For the last term in \eqref{eq:pU}, we first note that for $u\in U_1$,
	$u+e_ih$ cannot be in $U_3$ 	if $i\notin E$. So 
	\begin{align}\label{eq:u31}
		\sup\limits_{u\in U_1,u+e_ih\in U_3,  i\notin E}J(u,u+e_ih)=0.
	\end{align} 
	If $u\in U_1$ and $u+e_i h\in U_3$ where $i\in E$,
	$$|h|\geq|u^i+h-x^i|-|x^i-Q^i|-|u^i-Q^i|\ge|x^i-y^i|/2-3r_t/2\ge|x^i-y^i|/8.$$ 
	Therefore,  by \eqref{a:kappa}, \eqref{a:J} and ({\bf WS}),
	$$\sup_{u\in U_1,u+e_ih\in U_3, i\in E}J(u,u+e_ih)\leq \frac{\uC8^{1+\ua}\kappa_0}{|x^i-y^i|\phi(|x^i-y^i|)}\leq c_6\frac{1}{t\phi^{-1}(t)}\left(1\wedge\frac{t\phi^{-1}(t)}{|x^i-y^i|\phi(|x^i-y^i|)}\right).$$
	Note that for  $u\in U_1$ with $k\in E\setminus\{i\}$, $|u^k-y^k|\geq|x^k-y^k|-|u^k-x^k|\geq|x^k-y^k|-2r_t\geq|x^k-y^k|/2$. 
	Then by \eqref{eq:hke}, for $u\in U_1$ and $s\in (t/2, t)$,
	\begin{align}\label{eq:u32}
		&\sum_{i\in E}\Bigg(\sup_{\substack{u\in U_1\\ u+e_ih\in U_3}}\hspace{-.1in}J(u,u+e_ih)
		\cdot\int_{\bR}p(s,u+e_ih,y)dh\Bigg)\nn\\
		\le &  \sum_{i\in E}
		\Bigg(\frac{c_6}{t\phi^{-1}(t)}\Big(1\wedge\frac{t\phi^{-1}(t)}{|x^i-y^i|\phi(|x^i-y^i|)}\Big)
		\nn\\
		&\cdot \ C_1[\phi^{-1}(s)]^{-d+1}\prod_{ k\notin E}1
		\prod_{k\in E\setminus\{i\}}\Big(1\wedge\frac{2^{1+\ua}s\phi^{-1}(s)}{|x^k-y^k|\phi(|x^k-y^k|)}\Big)\int_\bR\Big([\phi^{-1}(s)]^{-1}\wedge\frac{s}{|h|\phi(|h|)}\Big)dh\Bigg)\nn  \\
		\leq&\ c_7t^{-1}[\phi^{-1}(t)]^{-{d}}\prod_{k=1}^d\left(1\wedge\frac{t\phi^{-1}(t)}{|x^k-y^k|\phi(|x^k-y^k|)}\right).	
	\end{align}
	For the last inequality, we used the fact that $\int_\bR\Big([\phi^{-1}(s)]^{-1}\wedge\frac{s}{|h|\phi(|h|)}\Big)dh<\infty$ and  $|x^k-y^k|<4V^{-1}(\sqrt{at})$ for $k\notin E$.
	By \eqref{eq:hke} and \autoref{L:exit} with \eqref{eq:u31} and \eqref{eq:u32}, we obtain
	\begin{align}\label{eq:u3}
		&\int_{0}^{t/2}\int_{U_1}p_{U_1}(s,x,u)\Bigg(\sum_{i=1}^d\int\limits_{\{h\in \bR: u+e_ih\in U_3\} }\hspace{-.2in}p_U(t-s,u+e_ih,y)dh\cdot\sup_{\substack{u\in U_1\\ u+e_ih\in U_3}}J(u,u+e_ih)\Bigg)duds\nn\\
		&\leq \int_{t/2}^t\int_{U_1}p_{U_1}(t-s,x,u)
		\left(\sum_{i\in E}\sup_{\substack{u\in U_1\\ u+e_ih\in U_3}}J(u,u+e_ih)
		\cdot\int_{\bR}p(s,u+e_ih,y)dh\right)duds\nn\\
		&\leq  c_7t^{-1}[\phi^{-1}(t)]^{-{d}}\prod_{k=1}^d\left(1\wedge\frac{t\phi^{-1}(t)}{|x^k-y^k|\phi(|x^k-y^k|)}\right)\int_{t/2}^t\int_{U_1}p_{U_1}(t-s,x,u)duds\nn\\
		&\leq c_8p(t,x,y)t^{-1}\int_0^{t/2}\Pp_x(\tau_{U_1}>s)ds\leq c_8p(t,x,y)t^{-1}\E_x[\tau_{U_1}]\nn\\
		&\leq c_9C_{\ref{L:exit}.1}t^{-1}V(r_t)V(\delta_D(x))p(t,x,y)\leq c_8\frac{V(\delta_D(x))}{\sqrt{t}}p(t ,x,y).
	\end{align}
	Applying \eqref{eq:u1}, \eqref{eq:u2}, \eqref{eq:u3} and \eqref{eq:hke} to \autoref{lemma:pU}, it  gives the result.
\end{proof}

\medskip
We are now ready to prove \autoref{t:nmain}.
\medskip

{\it Proof of \autoref{t:nmain}.}
By \autoref{prop:upper1}, the semigroup and symmetry property of Dirichlet heat kernel $p_D(t,x,y)$ yield that
\begin{align*}
	p_D(t,x,y)&=\int_{\Rd}p_D(t/2,x,z)p_D(t/2,z,y)dz\nn\\
	&\leq C_{\ref{prop:upper1}}^2\Psi(t/2,x)\Psi(t/2,y)\int_{\Rd}p(t/2,x,z)p(t/2,z,y)dz\nn\\
	&=C_2\Psi(t/2,x)\Psi(t/2,y)p(t,x,y).
\end{align*}
\hfill 
\rule{0.5em}{0.5em}

\section{Dirichlet lower bound estimate}\label{sec:lower}
{In this section, we first obtain an interior near-diagonal lower bound of $p_U(t,x,y)$ on an open set $U$, see \autoref{prp:lowersec2}. Then  under the condition $(\mathbf D_\gamma)$ for $\gamma\in(0,1]$, we discuss a preliminary lower bound in \autoref{prp:lower1}. The boundary decay result in \autoref{L:intpD} plays a vital role in proving the sharp lower bound estimate for the Dirichlet heat kernel for small time in \autoref{t:nmain2}(1).}
\subsection{The preliminary lower bound estimate} 
{For any fixed $r>0,x_0\in\Rd$, let $B:=B(x_0,r)$. In the following, we discuss the  lower bound for $p_B(t,x,y)$.}
\begin{proposition}\label{prp:pBlower}
	For any $\zeta_1>0$, there exist $\zeta_2:=\zeta_2(\phi, \zeta_1, d), \ \zeta_3:=\zeta_3(\phi, \zeta_1, d)>0$ such that for any $t\leq \phi(\zeta_1r)$ and for  any $x,y\in B(x_0,\zeta_2\phi^{-1}(t))\subset B$, 
	$$p_B(t,x,y)\geq \zeta_3[\phi^{-1}(t)]^{-d}.$$
\end{proposition}
\begin{proof}\hspace{-.05in}.
	For any $t>0$ and $0\leq h\in L^1(\Rd)\cap L^\infty(\Rd)$, by strong Markov property, we have that
	\begin{align}\label{eq:PtSMP}
		P_th(x)=&P_t^Bh(x)+\E_{x}[P_{t-\tau_B}h(X_{\tau_{B}});t>\tau_B]		\leq P_t^Bh(x)+\sup_{s\in(0,t)}\sup_{z\in B^c}P_sh(z).
	\end{align}
	Let  $f,g\in L^1(\Rd)\cap L^\infty(\Rd)$ be nonnegative functions.
	Replacing $t$ by $t/2$, and $h$ by $P_{t/2}f, P_{t/2}^Bg$ in \eqref{eq:PtSMP}, respectively yield that
	\begin{align*}
		P_tf(x)=P_{t/2}P_{t/2}f(x)&\leq P_{t/2}^BP_{t/2}f(x)+\sup_{s\in(0,t/2)}\sup_{z\in B^c}P_sP_{t/2}f(z);\\
		P_{t/2}P^B_{t/2}g(x)&\leq P_{t}^Bg(x)+\sup_{s\in(0,t/2)}\sup_{z\in B^c}P_sP^B_{t/2}g(z).
	\end{align*}
	Hence, we obtain that
	\begin{align*}
		\langle P_tf,g\rangle\leq&\langle P_{t/2}^BP_{t/2}f,g\rangle+\sup_{s\in(0, t/2)}\sup_{z\in B^c}P_sP_{t/2}f(z)\|g\|_{L^1}\\
		=&\langle f,P_{t/2}P_{t/2}^Bg\rangle+\sup_{s\in(t/2,t)}\sup_{z\in B^c}P_sf(z)\|g\|_{L^1}\\
		\leq &\langle f,P_{t}^Bg\rangle+\sup_{s\in(0,t/2)}\sup_{z\in B^c}P_sP^B_{t/2}g(z)\|f\|_{L^1}+\sup_{s\in(t/2,t)}\sup_{z\in B^c}P_sf(z)\|g\|_{L^1}\\
		\leq&\langle P_{t}^Bf,g\rangle+\sup_{s\in(t/2,t)}\sup_{z\in B^c}P_{ s}g(z)\|f\|_{L^1}+\sup_{s\in(t/2,t)}\sup_{z\in B^c}P_sf(z)\|g\|_{L^1}.
	\end{align*}
	For any Borel set $E\subset B$, the above inequality holds for all {nonnegative functions $ f,g\in L^1(\Rd)\cap L^\infty(\Rd)$} with $\text{Supp}(f), \text{Supp}(g)\subset E$.
	So for any $t>0$ and $x,y\in E$,
	\begin{equation}\label{eq:ppB}
		p(t,x,y)\leq p_B(t,x,y)+2\sup_{s\in(t/2,t)}\sup_{w\in E, z\in B^c}p(s,w,z).
	\end{equation}
	Let $E=B(x_0,(r\wedge\phi^{-1}(t))/2)$. For any $t\leq \phi(\zeta r)$, $w\in E$ and $z\in B^c$, 
	$|w-z|\geq \frac{r}{2}\geq\frac{\phi^{-1}(t)}{2\zeta},$
	and there exists $1\leq k\leq d$ such that 
	$$|w^k-z^k|\geq\frac{\phi^{-1}(t)}{2\sqrt{d}\zeta}.$$
	Now applying \eqref{eq:ppB} with \eqref{eq:hke} and ({\bf WS}), for any $x,y\in B(x_0,(r\wedge\phi^{-1}(t))/2)$
	\begin{align*}
		p_B(t,x,y)\geq &c_1[\phi^{-1}(t)]^{-d}-2\sup_{s\in(t/2,t)}\sup_{w\in E, z\in B^c}p(s,w,z)\\
		\geq&c_1[\phi^{-1}(t)]^{-d}-c_2[\phi^{-1}(t/2)]^{-d}\frac{t\phi^{-1}(t)}{|w^k-z^k|\phi(|w^k-z^k|)}\\
		\geq&c_1[\phi^{-1}(t)]^{-d}-c_2\uC^{-d}2^{-\ua d}[\phi^{-1}(t)]^{-d}\frac{1}{(2\sqrt d\zeta)^{-(\la+1)}\wedge (2\sqrt d\zeta)^{-(\ua+1)}}\\
		\geq&\frac{c_1}{2}[\phi^{-1}(t)]^{-d}.
	\end{align*}
	The last inequality holds by selecting the value of $\zeta$  such that
	$$(2\sqrt d\zeta)^{-(\la+1)}\wedge (2\sqrt d\zeta)^{-(\ua+1)}=c_1^{-1}c_2\uC^{-d}2^{-\ua d+1}.$$
	According to $\zeta$ and $\zeta_1$, we can observe the two cases, $\zeta_1\le \zeta$ and $\zeta_1>\zeta$.\\
	\textit{Case 1.} If $\zeta_1\leq \zeta$, take $\zeta_2:=\frac{1}{2}(1\wedge \zeta_1^{-1})$. Then for any $t\leq \phi(\zeta_1r)\leq \phi(\zeta r)$,
	$$\zeta_2\phi^{-1}(t)=\frac{1}{2}\left(\phi^{-1}(t)\wedge \zeta_1^{-1}\phi^{-1}(t)\right)\leq \frac{1}{2}(\phi^{-1}(t)\wedge r).$$
	Thus, $B(x_0,\zeta_2\phi^{-1}(t))\subset B(x_0,(r\wedge\phi^{-1}(t))/2)$, the result gets proved by letting $\zeta_3=\frac{c_1}{2}$.\\
	\textit{Case 2.} If $\zeta_1>\zeta$, we then take 
	$$\zeta_2:=\frac{\zeta}{2\zeta_1}(1\wedge \zeta_1^{-1}).$$
	For any $t\leq \phi(\zeta_1r)$, we first have $B(x_0,\zeta_2\phi^{-1}(t))\subset B(x_0,(r\wedge\phi^{-1}(t))/2)$ since {$\zeta_2\leq\frac{1}{2}$ and
		$$\zeta_2\phi^{-1}(t)\leq \frac{\zeta}{2\zeta_1}(1\wedge \zeta_1^{-1})\zeta_1r\leq\frac{1}{2}  \zeta\zeta_1^{-1}r\leq \frac{1}{2}r.$$}
	Now we define $$n=n(r):=\left[\frac{\phi(\zeta_1r)}{\phi(\zeta r)}\right]+1,$$
	where $[x]$ represents the largest integer that is less than or equal to $x$. Note that by ({\bf WS}), we have two-sided bounds for $n$ that is independent of $r$, 
	\begin{equation}
		\label{eq:n}
		\lC\left(\frac{\zeta_1}{\zeta}\right)^{\la}\vee 1\leq n\leq\uC\left(\frac{\zeta_1}{\zeta}\right)^\ua+1.
	\end{equation}The semigroup property implies that for any $t\leq \phi(\zeta_1r)$ and $x,y\in B(x_0,\zeta_2\phi^{-1}(t))=:B_0$,
	\begin{align*}
		p_B(t,x,y)=&\int_B\cdots\int_Bp_B(\tfrac{t}{n},x,z_1)p_B(\tfrac{t}{n},z_1,z_2)\cdots p_B(\tfrac{t}{n},z_{n-1},y)dz_1dz_2\cdots dz_{n-1}\\
		\geq&\int_{B_0}\cdots\int_{B_0}p_B(\tfrac{t}{n},x,z_1)p_B(\tfrac{t}{n},z_1,z_2)\cdots p_B(\tfrac{t}{n},z_{n-1},y)dz_1dz_2\cdots dz_{n-1}\\
		\geq &\left(\frac{c_1}{2}[\phi^{-1}(t)]^{-d}\right)^n|B(x_0,\zeta_2\phi^{-1}(t))|^{n-1}\\
		=&\left(\frac{c_1}{2}\right)^{n}[\phi^{-1}(t)]^{-dn}\left(\frac{\pi^{d/2}}{\Gamma(1+\frac{d}{2})}\right)^{n-1}\zeta_2^{d(n-1)}[\phi^{-1}(t)]^{d(n-1)}\\
		\geq& \left(\frac{c_1}{2}\right)^{n}\left(\frac{\pi^{d/2}}{\Gamma(1+\frac{d}{2})}\right)^{n-1}\zeta_2^{d(n-1)}[\phi^{-1}(t)]^{-d}=:\zeta_3[\phi^{-1}(t)]^{-d},
	\end{align*}
	where, according to \eqref{eq:n}, $\zeta_3$ depends on $\phi, \zeta_1, d$.
\end{proof}

\medskip
As an immediate application of \autoref{prp:pBlower}, we arrive at the following  interior near-diagonal lower bound for the Dirichlet heat kernel.
\begin{proposition}\label{prp:lowersec2}
	For any open set $U\subset \Rd$ and any $\zeta_1>0$, there exist constants $\zeta_2:=\zeta_2(\phi, \zeta_1, d), \zeta_3:=\zeta_3(\phi, \zeta_1, d)>0$ such that for any $(t,x,y)\in(0,\infty)\times U\times U$ with $\delta_U(x)\wedge\delta_U(y)\geq \zeta_1^{-1}\phi^{-1}(t)$ and $|x-y|<\zeta_2\phi^{-1}(t)$, we have
	$$p_U(t,x,y)\geq \zeta_3[\phi^{-1}(t)]^{-d}.$$
\end{proposition}

\begin{proof}\hspace{-.05in}.
	For any fixed $t>0$, let $r:=\zeta_1^{-1}\phi^{-1}(t)$. For $x\in U$ with $\delta_U(x)\geq r$, $B:=B(x,r)\subset U$. 
	By \autoref{prp:pBlower}, there exists ${\zeta_2}, \zeta_3>0$ such that 
	$$ p_B(t,y,z)\geq \zeta_3[\phi^{-1}(t)]^{-d}\qquad\text{ for any }y,z\in B(x,{\zeta_2}\phi^{-1}(t)) \subset B.$$
	
	Thus, for any $x,y\in U$ with $\delta_U(x)\wedge\delta_U(y)\geq \zeta_1^{-1}\phi^{-1}(t)$ and $|x-y|<{\zeta_2}\phi^{-1}(t)$, 
	$$y\in B(x,{\zeta_2}\phi^{-1}(t)) \subset B\subset U,$$
	and we have $p_U(t,x,y)\geq p_B(t,x,y)\geq \zeta_3[\phi^{-1}(t)]^{-d}.$
\end{proof}

\medskip
\begin{lemma}\label{inftaulower}
	Let $a,b$ and $T$ be positive constants. Then there exists a constant $C_{\ref{inftaulower}}=C_{\ref{inftaulower}}(a,b,\phi,T, d)>0$ such that for any $t\leq T$,
	\begin{equation*}
		\inf_{ y\in B(x,b\phi^{-1}(t))}\Pp_y\left(\tau_{B(x,2b\phi^{-1}(t))}>at\right)\geq C_{\ref{inftaulower}}.
	\end{equation*}
\end{lemma}
\begin{proof}\hspace{-.05in}.
	For any $y\in B(x,b\phi^{-1}(t))$, $$B(y,b\phi^{-1}(t))\subset B(x,2b\phi^{-1}(t)).$$ 
	By \autoref{prp:pBlower} with $r=b\phi^{-1}(T)$, $\zeta_1=r^{-1}\phi^{-1}(aT)$ and $ B:=B(y,b\phi^{-1}(t))$
	there exist $\zeta_2=\zeta_2(a,b, \phi, T, d), \zeta_3=\zeta_3(a,b, \phi, T, d)>0$ such that $B(y,\zeta_2\phi^{-1}(at))\subset B$ and
	$$ p_B(at,y,z)\ge \zeta_3[\phi^{-1}(at)]^{-d}\qquad \text{ for }z\in B(y,\zeta_2\phi^{-1}(at)).$$
	Hence, for any $y\in B(x,b\phi^{-1}(t))$,
	\begin{align*}
		\Pp_y(\tau_{B(x,2b\phi^{-1}(t))}>at)=&\int_{B(x,2b\phi^{-1}(t))}p_{B(x,2b\phi^{-1}(t))}(at,y,u)du\\
		\ge&\int_{ B}p_{B}(at,y,u)dy\ge\int_{B(y,\zeta_2\phi^{-1}(at))}p_B(at,y,u)du\\	
		\ge& \zeta_3[\phi^{-1}(at)]^{-d}|B(y,\zeta_2\phi^{-1}(at))|	=\zeta_3\zeta_2^d|B(0,1)|>0.
	\end{align*}
\end{proof}

\medskip
The following result is obtained in \cite[Propostion 4.1]{KW22}.
\begin{proposition}\label{p:exit}
	For any $A>0$ and $B\in (0, 1)$, there exists ${ \varrho}={\varrho}(A, B, \phi)\in (0, 2^{-1})$ such that for every $r>0$ and $x\in \Rd$,
	\begin{align*}
		\Pp_x(\tau_{B(x, Ar)}<{\varrho} \phi(r))=\Pp_x\Big(\sup_{s\le {\varrho}\phi(r) }|X_s-X_0|>Ar \Big)\le B.
	\end{align*}
\end{proposition}
Now we prove the preliminary lower bound of Dirichlet heat kernel, and this gives the sharp interior lower bound.
\begin{proposition}\label{prp:lower1}
	Assume that $U$ is an open set satisfying condition $(\mathbf D_\gamma)$ for some $\gamma\in(0,1]$. For any $ a, t>0$ and $x,y\in U$ with $\delta_U(x)\wedge \delta_U(y)\geq a\phi^{-1}(t)$, there exists a constant $C_{\ref{prp:lower1}}=C_{\ref{prp:lower1}}( a,\phi,\kappa_0,\gamma, d)>0$ such that 
	\begin{equation*}
		p_U(t,x,y)\ge C_{\ref{prp:lower1}}p(t,x,y).
	\end{equation*}
\end{proposition}
\begin{proof}\hspace{-.05in}.
	With $\delta_{U}(x)\wedge \delta_{ U}(y)\geq a\phi^{-1}(t)$, by \autoref{prp:lowersec2}, there exists $\zeta_2=\zeta_2(a,\phi, d), \zeta_3=\zeta_3(a,\phi, d)>0$ such that if $|x-y|<\zeta_2\phi^{-1}(t)$,
	$$p_{U}(t,x,y)\geq \zeta_3[\phi^{-1}(t)]^{-d}.$$
	It suffices to consider the case that $\delta_{ U}(x)\wedge \delta_{U}(y)\geq a\phi^{-1}(t)$ and $|x-y|\ge \zeta_2\phi^{-1}(t)$. Without loss of generality, we assume $\zeta_2\le a$. Now let $0<b\leq \zeta_2/2$ and $0<\rho<1$ be constants that are determined later. Then by the semigroup property,
	\begin{align}\label{eq:pUlower1}
		p_U(t,x,y)=&\int_Up_U(\rho t,x,z)p_U((1-\rho)t,z,y)dz\ge\int_{B(y,b\phi^{-1}(t))}p_U(\rho t,x,z)p_U((1-\rho)t,z,y)dz\nn\\
		\ge&\inf_{z\in B(y,b\phi^{-1}(t))}p_U((1-\rho)t,z,y)\Pp_x\left(X_{\rho t}^U\in B(y,b\phi^{-1}(t))\right).
	\end{align}
	Note that for $z\in B(y,b\phi^{-1}(t))$, $\delta_U(z)\ge(a-b)\phi^{-1}(t)\ge\frac{a}{2}\phi^{-1}(t)$. 
	Then by \autoref{prp:lowersec2}, there exists $\wt\zeta_2=\wt\zeta_2(a,d, \phi)>0$ such that if
	\begin{equation}\label{brho1}
		|y-z|\le b\phi^{-1}(t)<\wt\zeta_2 \lC^{1/\la}(1-\rho)^{1/\la}\phi^{-1}(t)<\wt\zeta_2\phi^{-1}((1-\rho)t),
	\end{equation}
	for some small $b, \rho>0$,  there exists $\wt\zeta_3=\wt\zeta_3(a, d, \phi)>0$ satisfying that
	\begin{equation}
		\label{eq:infpU}
		\inf_{z\in B(y,b\phi^{-1}(t))}p_U((1-\rho)t,z,y)\ge \wt\zeta_3[\phi^{-1}((1-\rho)t)]^{-d}\ge \wt\zeta_3\uC^{-d/\ua}(1-\rho)^{-d/\ua}[\phi^{-1}(t)]^{-d}.
	\end{equation}
	Since	$U$ satisfies condition $(\mathbf D_\gamma)$ and $\delta_U(x)\wedge \delta_U(y)\ge a\phi^{-1}(t)$, there exists a permutation $\{i_1,\dots,i_d\}$ of $\{1,\dots,d\}$ such that 
	$$B(\xi_{(k)},\gamma a\phi^{-1}(t))\subset U,\qquad\text{ for }k=1,\dots,d.$$
	Choose $b$ small enough such that 
	\begin{equation}\label{brho2}
		b<\gamma a/2.
	\end{equation} 
	For any $k=1,\dots, d$, define $r_k:=\frac{b\phi^{-1}(t)}{2^{d-k}\sqrt d}$ and
	$$Q_k:=Q( \xi_{(k)},r_k)\subset B_k:=B(\xi_{(k)},\sqrt d r_k)\subset U,$$
	where $Q(x,r)$ is  a cube centered at $x$ with  side length $r$.
	By the semigroup property, 
	\begin{align}\label{eq:pUlower2}
		&\Pp_x\left(X_{\rho t}^U\in B(y,b\phi^{-1}(t))\right)\nn\\
		\ge&\int_{Q_d}\cdots\int_{Q_1}p_U\left(\frac{\rho t}{d}, x,z_1\right)p_U\left(\frac{\rho t}{d},z_1,z_2\right)\cdots p_U\left(\frac{\rho t}{d}, z_{d-1},z_d\right)dz_1\cdots dz_d\nn\\
		\ge&\ \Pp_x\left(X_{\rho t/d}^U\in Q_1\right)\prod_{k=1}^{d-1}\inf_{z_k\in Q_k}\Pp_{z_k}\left(X_{\rho t/d}^U\in Q_{k+1}\right).
	\end{align}
	Let $r:=b\phi^{-1}(t)/\sqrt d$. Then $r_k/2\ge 2^{-d}r$ for any  $1\le k\le d-1$ and $t\le \uC d^{\ua/2}b^{-\ua} \phi(r)$ by ({\bf WS}).
	Therefore  for any $w\in Q_k$, 
	$$\Pp_w\left(\tau_{B(w,r_k/2)}\ge \frac{\rho t}{2d}\right)\ge\Pp_w\left(\tau_{B(w,2^{-d}r)}\geq\frac{\rho t}{2d}\right)\ge\Pp_w\left(\tau_{B(w,2^{-d}r)}\geq\frac{\rho\uC d^{\ua/2-1}}{2b^\ua} \phi(r)\right), 
	$$
	and by \autoref{p:exit}, we can select $ \rho$ small enough such that 
	\begin{equation}\label{brho3}
		\Pp_w\left(\tau_{B(w,r_k/2)}\ge \frac{\rho t}{2d}\right)\ge\frac{1}{2}.
	\end{equation}
	Let $\ol Q_k:=Q( \xi_{(k)},r_k/2)$. By  \eqref{brho3}, strong Markov property and L\'evy system in \eqref{eq:LSd}, for any $v\in Q_k$, we have that 
	\begin{align*}
		\Pp_v(X_{\rho t/d}\in Q_{k+1})\ge&\ 2^{-1}\Pp_v \left( X\text{ hits }\ol Q_{k+1} \text{ by time }\frac{\rho t}{2d}   \right)\\
		\ge &\ c_1\E_v\left[\int_0^{\frac{\rho t}{2d} \wedge \tau_{Q(v,r_k)}}\int_{\ol Q_{k+1}}\frac{1}{|X_s-u|\phi(|X_s-u|)}m(du)ds\right],
	\end{align*}
	where $m(du)$ is the measure on $\sum_{i=1}^d\bR$ restricted only on each coordinate.
	By the same discussion as \eqref{brho3},
	$$\E_v\left[\frac{\rho t}{2d} \wedge \tau_{Q(v,r_k)}\right]\ge \frac{\rho t}{2d}\Pp_v\left(\tau_{Q(v,r_k)}>\frac{\rho t}{2d}\right)\ge \frac{\rho}{4d}t.$$
	Suppose that $|x^{i_{k+1}}-y^{i_{k+1}}|\ge r_k$. Then for any $w\in Q(\xi_{(k)},2r_k)$ and $u\in Q( \xi_{(k+1)},2r_k)$, 	by \eqref{eq:rel},
	\begin{align*}
		|w^{i_{k+1}}-u^{i_{k+1}}|\leq&\ |x^{i_{k+1}}-y^{i_{k+1}}|+|x^{i_{k+1}}-w^{i_{k+1}}|+|y^{i_{k+1}}-u^{i_{k+1}}|\\
		=&\ |x^{i_{k+1}}-y^{i_{k+1}}|+| \xi_{(k)}^{i_{k+1}}-w^{i_{k+1}}|+| \xi_{(k+1)}^{i_{k+1}}-u^{i_{k+1}}|\leq 5|x^{i_{k+1}}-y^{i_{k+1}}|.
	\end{align*}
	Otherwise,  $|x^{i_{k+1}}-y^{i_{k+1}}|\le r_k$, $|w^{i_{k+1}}-u^{i_{k+1}}|\le 5r_k$. Therefore, for any $v\in Q_k$ and $k\in \{1,\dots, d-1\}$, 
	\begin{align}\label{eq:Qk}
		\Pp_v\left(X_{\rho t/d}\in Q_{k+1}\right)\ge&\ c_2\E_v\left[\frac{\rho t}{2d} \wedge \tau_{Q(v,r_k)}\right]\left(\frac{1}{t}\wedge \frac{\phi^{-1}(t)}{|x^{i_{k+1}}-y^{i_{k+1}}|\phi(|x^{i_{k+1}}-y^{i_{k+1}}|)}\right)\nn\\
		\ge &\ c_3\left(1\wedge \frac{t\phi^{-1}(t)}{|x^{i_{k+1}}-y^{i_{k+1}}|\phi(|x^{i_{k+1}}-y^{i_{k+1}}|)}\right).
	\end{align}
	In a similar way to obtain \eqref{eq:Qk}, we have that
	\begin{equation}\label{eq:Q1}
		\Pp_x\left(X_{\rho t/d}^U\in Q_1\right)\ge c_4\left( 1\wedge\frac{t\phi^{-1}(t)}{|x^{i_1}-y^{i_1}|\phi(|x^{i_1}-y^{i_1}|)}\right).
	\end{equation}	
	By choosing $b,\rho$ small enough  that \eqref{brho1}, \eqref{brho2} and \eqref{brho3} are satisfied,
	and by plugging \eqref{eq:Qk}--\eqref{eq:Q1} into \eqref{eq:pUlower2},  we arrive at our result by \eqref{eq:infpU} and \eqref{eq:pUlower2} with \eqref{eq:pUlower1}.
\end{proof}

\subsection{The lower bound} We now obtain the boundary decay using \autoref{L:exit} and \autoref{inftaulower}.
For  fixed $T>0$,  define
$$\wh a_T:=\wh a_{T,R}:=\frac{R}{\sqrt{1+\Lambda^2}\phi^{-1}(T)}\quad  \text{ and } \quad \wh r_t:=\wh a_T\phi^{-1}(t)\ \text{ for any }t\in[0,T]. $$
By \eqref{e:com},  there exists $C_0$ such that  
$\frac{V(\wh r_t)^2}{t}\le C_0$.
Using $C_{\ref{L:exit}.1}, C_{\ref{L:exit}.2}$ and $C_V$, let 
\begin{align}\label{d:nota}
	\lambda:=8\Lambda\sqrt{1+\Lambda^2}\vee \left(\frac{6C_{\ref{L:exit}.1} C_V}{C_0C_{\ref{L:exit}.2}}\right)^{2/\la}\quad  \text{ and } \quad \lambda_1:=2\lambda\sqrt{1+\Lambda^2}.
\end{align}

For any $x\in D$, let $Q_x\in\partial D$ such that $\delta_D(x)=|x-Q_x|$, and define sets
\begin{align}\label{d:E}
	U_x:=&D_{Q_x}( \wh r_t,\wh r_t)\setminus D_{Q_x}( \wh r_t,\lambda^{-1}\wh r_t)\nn\\
	E_x:=&\begin{cases}
		B(x,\lambda_1^{-1}\wh r_t),&\quad x\notin D_{Q_x}(\lambda^{-1}\wh r_t,\lambda^{-1}\wh r_t);\\
		\bigcup_{z\in U_x}B(z,\lambda_1^{-1}\wh r_t),&\quad x\in D_{Q_x}(\lambda^{-1}\wh r_t,\lambda^{-1}\wh r_t).
	\end{cases}
\end{align}

\begin{lemma}\label{L:intpD}
	For any $T>0$, there exists a constant $C_{\ref{L:intpD}}=C_{\ref{L:intpD}}(\phi, \kappa_0,\kappa_1, \eta, R, \Lambda, T, d)>0$ such that for any $t\le T$ and $x\in D$,
	\begin{equation*}
		\int_{E_x}p_D(t/3,x,z)dz\geq C_{\ref{L:intpD}}\Psi(t,x)
	\end{equation*}
	where $E_x$ is the set  defined in \eqref{d:E}.
\end{lemma}
\begin{proof}\hspace{-.05in}.
	By the definition of $(R, \Lambda)$, we may assume that $\Lambda>2$.
	So for fixed $T>0$, 
	$$\wh a_T:=\wh a_{T,R}:=\frac{R}{\sqrt{1+\Lambda^2}\phi^{-1}(T)}\le \frac{R}{2\phi^{-1}(T)},$$
	and $\wh r_t<{R}/{2}$ for  $t\le T$. 
	If $x\notin D_{Q_x}(\lambda^{-1}\wh r_t,\lambda^{-1}\wh r_t)$, then $\delta_D(x)\ge \lambda^{-1}\wh r_t$. Let $B_x:=B(x,\lambda_1^{-1}\wh r_t)$, then by \autoref{inftaulower}, 
	$$\int_{E_x}p_D(t/3,x,z)dz\ge\int_{B_x}p_{B_x}(t/3,x,z)dz=\Pp_x(\tau_{B_x}>t/3)\ge c_1.$$
	If $x\in D_{Q_x}(\lambda^{-1}\wh r_t,\lambda^{-1}\wh r_t)$, that is, $\delta_D(x)< \lambda^{-1}\wh r_t$,  Markov property implies that 
	\begin{align}\label{eq:xnotinDQ}
		\int_{E_x}p_D(t/3,x,z)dz= \		 \Pp_x(X^D_{t/3}\in E_x)
		\geq&\ \Pp_x\left(\sigma^D_{U_x}<t/3,\tau_{E_x}\circ\theta_{\sigma^D_{U_x}}>t/3\right)\nn\\
		\ge&\inf_{z\in U_x}\Pp_z\left(\tau_{E_x}>t/3\right)\Pp_x\left(\sigma^D_{U_x}<t/3\right)
	\end{align}
	where $\sigma^D_{U_x}:=\inf\{s\ge 0: X_s^D\in U_x\}$ is the first hitting time of $U_x$.
	For any $z\in U_x$, $\delta_D(z)\ge \lambda^{-1}\wh r_t$ and $ B_z:=B(z,\lambda_1^{-1}\wh r_t)\subset E_x\subset D$, so  \autoref{inftaulower} implies that
	\begin{equation}\label{eq:tauEx}
		\inf_{z\in U_x}\Pp_z\left(\tau_{E_x}>t/3\right)\ge \inf_{z\in U_x}\Pp_z\left(\tau_{B_z}>t/3\right)\ge c_2.
	\end{equation}
	
	Applying \autoref{L:exit} with $s=\lambda^{-1}\wh r_t$ and  with $s=\wh r_t$ for $\delta_D(x)\le \lambda^{-1}\wh r_t$, respectively, we obtain 
	$$\E_x[\tau_{D_{Q_x}(\lambda^{-1}\wh r_t,\lambda^{-1}\wh r_t)}]
	\le C_{\ref{L:exit}.1}{V(\lambda^{-1}\wh r_t)V(\delta_D(x))}\le C_{\ref{L:exit}.1}C_V \lambda^{-\la/2}{V(\wh r_t)V(\delta_D(x))}$$
	$$\Pp_x\left(X_{\tau_{D_{Q_x}(\lambda^{-1}\wh r_t,\lambda^{-1}\wh r_t)}}\in D_{Q_x}( \wh r_t,\wh r_t)\setminus D_{Q_x}( \wh r_t,\lambda^{-1}\wh r_t)\right)\ge  C_{\ref{L:exit}.2}\frac{V(\delta_D(x))}{V(\wh r_t)}.\qquad $$
	Therefore,
	\begin{align*}
		&\Pp_x\left(\sigma^D_{U_x}<t/3\right)\ge\ \Pp_x\left(\tau_{D_{Q_x}(\lambda^{-1}\wh r_t, \ \lambda^{-1}\wh r_t)}<t/3,X^D_{\tau_{D_{Q_x}(\lambda^{-1}\wh r_t,\lambda^{-1}\wh r_t)}}\in U_x\right)\nn\\
		=&\ \Pp_x\left(X_{\tau_{D_{Q_x}(\lambda^{-1}\wh r_t,\lambda^{-1}\wh r_t)}}\in D_{Q_x}( \wh r_t,\wh r_t)\setminus D_{Q_x}( \wh r_t,\lambda^{-1}\wh r_t)\right)\nn\\
		&-\Pp_x\left(X_{\tau_{D_{Q_x}(\lambda^{-1}\wh r_t,\lambda^{-1}\wh r_t)}}\in D_{Q_x}( \wh r_t,\wh r_t)\setminus D_{Q_x}( \wh r_t,\lambda^{-1}\wh r_t),\ \tau_{D_{Q_x}(\lambda^{-1}\wh r_t,\lambda^{-1}\wh r_t)}\ge t/3\right)\nn\\
		\ge&\ \Pp_x\left(X_{\tau_{D_{Q_x}(\lambda^{-1}\wh r_t,\lambda^{-1}\wh r_t)}}\in D_{Q_x}( \wh r_t,\wh r_t)\setminus D_{Q_x}( \wh r_t,\lambda^{-1}\wh r_t)\right)-\Pp_x\left(\tau_{D_{Q_x}(\lambda^{-1}\wh r_t,\lambda^{-1}\wh r_t)}\ge t/3\right)\nn\\
		\ge&\ C_{\ref{L:exit}.2}\frac{V(\delta_D(x))}{V(\wh r_t)}-\frac{\E_x[\tau_{D_{Q_x}(\lambda^{-1}\wh r_t,\lambda^{-1}\wh r_t)}]}{t/3}
		\ge \ C_{\ref{L:exit}.2}\frac{V(\delta_D(x))}{V(\wh r_t)}- \frac{3C_{\ref{L:exit}.1}C_V}{\lambda^{\la/2}}\frac{V(\wh r_t)V(\delta_D(x))}{t}\nn\\
		=&\ \left({C_{\ref{L:exit}.2}}-\frac{3C_{\ref{L:exit}.1}C_V}{\lambda^{\la/2}}\frac{V(\wh r_t)^2}{t}\right)\frac{V(\delta_D(x))}{V(\wh r_t)}\ge \frac{C_{\ref{L:exit}.2}}{2}\frac{V(\delta_D(x))}{V(\wh r_t)}.
	\end{align*}
	The last inequality holds by \eqref{d:nota}.
	Since $V(\delta_D(x))\asymp \sqrt{\phi(\delta_D(x))}$ and $V(\wh r_t)\asymp \sqrt t$ by \eqref{e:com}, we conclude that
	\begin{equation}\label{eq:sigmaUx}
		\Pp_x(\sigma_{U_x}^D<t/3)\ge c_3\sqrt{\frac{\phi(\delta_D(x))}{t}}.
	\end{equation}
	Therefore, we arrive our result by plugging \eqref{eq:tauEx} and \eqref{eq:sigmaUx} into \eqref{eq:xnotinDQ}.
\end{proof}

\vspace{.1in}

Now we are ready to prove \autoref{t:nmain2}(1).
\vspace{.1in}

\begin{proof}{\it of \autoref{t:nmain2}(1).}
	By the semigroup property and \autoref{L:intpD}, for any $x,y\in D$, 
	\begin{align}\label{eq:pDlower}
		p_D(t,x,y)=&\int_{D\times D}p_D(t/3,x,u)p_D(t/3,u,v)p_D(t/3,v,y)dudv\nn\\
		\ge&\int_{E_x\times E_y}p_D(t/3,x,u)p_D(t/3,u,v)p_D(t/3,v,y)dudv\nn\\
		\ge&\inf_{u\in E_x,v\in E_y}p_D(t/3,u,v)\int_{E_x}p_D(t/3,x,u)du\cdot\int_{E_y}p_D(t/3,v,y)dv\nn\\
		\ge&\ C_{\ref{L:intpD}}^2\Psi(t,x)\Psi(t,y)\inf_{u\in E_x,v\in E_y}p_D(t/3,u,v)
	\end{align}
	where $E_x, E_y$ are defined in \eqref{d:E}.
	Note that for $(u,v)\in E_x\times E_y$, $|x^i-u^i|\vee|y^i-v^i|\leq 3\wh r_t$ for all $i\in\{1,\dots,d\}$. If $|x^i-y^i|\le 3 \wh r_t$, then $|u^i-v^i|\le |x^i-u^i|+|x^i-y^i|+|y^i-v^i|\le 9 \wh r_t$.
	Using ({\bf WS}), $\phi(|u^i-v^i|)\le \phi( 9\wh a_T\phi^{-1}(t))\asymp \phi(\phi^{-1}(t))=t$, and so
	\begin{equation}\label{eq:uvi1}
		|u^i-v^i|\phi(|u^i-v^i|)\le c_1 t\phi^{-1}(t).
	\end{equation}
	If $|x^i-y^i|> 3{\wh r_t}$, then $|u^i-v^i|\leq 3|x^i-y^i|$, and
	\begin{equation}\label{eq:uvi2}
		|u^i-v^i|\phi(|u^i-v^i|)\le c_2|x^i-y^i|\phi(|x^i-y^i|).
	\end{equation}
	Recall that  $\rho_{Q}(w):=w^{d}-\varphi_Q(\wt{w})$ for $w\in D$. Then by \eqref{comp:rho}  and  the fact that  $\lambda_1\ge 2\lambda$, for any $(u,v)\in E_x\times E_y$,
	$$\delta_D(u)\wedge \delta_D(v)\ge  \frac{\rho_{Q_x}(u)\wedge \rho_{Q_y}(v)}{\sqrt{1+\Lambda^2}}
	\ge \frac{ (\lambda^{-1}-\lambda_1^{-1})\wh r_t}{\sqrt{1+\Lambda^2}}
	\ge \frac{\wh a_T}{2\lambda\sqrt{1+\Lambda^2}}\phi^{-1}(t).$$
	Applying \autoref{prp:lower1} with $a= \frac{\wh a_T}{2\lambda\sqrt{1+\Lambda^2}}$, and then combining \eqref{eq:hke} with \eqref{eq:uvi1}--\eqref{eq:uvi2}, we obtain the lower bound of the last part in \eqref{eq:pDlower} as follows:
	\begin{align*}
		\inf_{u\in E_x,v\in E_y}p_D(t/3,u,v)\ge&\ C_{\ref{prp:lower1}}\inf_{u\in E_x,v\in E_y}p(t/3,u,v)\\
		\ge&\ C_{\ref{prp:lower1}}C_1^{-1}[\phi^{-1}(t)]^{-d}\prod_{i=1}^d\left(1\wedge \frac{t\phi^{-1}(t)}{|u^i-v^i|\phi(|u^i-v^i|)}\right)\\
		\ge&\ \frac{C_{\ref{prp:lower1}}}{C_1(c_1\vee c_2)^{d}}[\phi^{-1}(t)]^{-d}\prod_{i=1}^d\left(1\wedge \frac{t\phi^{-1}(t)}{|x^i-y^i|\phi(|x^i-y^i|)}\right)\ge  c_3p(t ,x,y).
	\end{align*}
\end{proof}

\section{Large time Dirichlet heat kernel estimates}\label{sec:largeDHKE}

In this section, we assume $D$ is a bounded $C^{1,1}$ open set satisfying condition $(\mathbf D_\gamma)$ for some $\gamma\in(0,1]$, and the jumping kernel $J$ satisfies \eqref{a:J}, $\bf{(K_\eta), (SD)}$. We aim to prove \autoref{t:nmain2}(2) by using the on-diagonal  estimate for large time in \autoref{eq:Ondia}.
\begin{lemma}\label{eq:Ondia}
	Suppose $D\subset \bR^d$ is a bounded open set.
	Then there are two positive constants $c_i=c_i(\phi, \text{diam}(D), d), i=1,2$ such that 
	$$p_D(t, x, y)\le c_1 e^{-c_2t} , \qquad\text{ for }(t,x,y)\in(1,\infty)\times D\times D.$$	
\end{lemma}
\begin{proof}\hspace{-.05in}.
	By the semigroup property and \eqref{eq:hke}, for all $(t,x,y)\in(1,\infty)\times D\times D$, 
	\begin{align}\label{eq:t}
		p_D(t, x, y)=\int_D p_D(t-1, x, z) p_D(1, z, y) dz\le 	C_1[\phi^{-1}(1)]^{-d}\int_D p_D(t-1, x, z)dz 
	\end{align}
	and for $t\in (n , n+1]\cap (2, \infty)$, there is $n\in \bN$ and 
	\begin{align}\label{eq:t-1}
		\int_D p_D(t-1, x, z)dz= \int_{D^{n+1}} p_D(1, x, x_1)&\cdots  p_D(1, x_{n-1},  x_n)\nn\\
		&  p_D(t-n, x_n, z)\ dz\ dx_n\cdots dx_1.
	\end{align}
	Let $d_0:=\text{diam}(D)$ and $p^\phi(1, x^1, y^1)$ be the transition density for $X^1$.
	Since $\{|z|\ge d_0\}\supset \{z=(z^1, \ldots, z^d): |z^1|\ge d_0\vee 1\}$, by \eqref{eq:hke} and ({\bf WS}), for $x\in D$,
	\begin{align*}
		\bP_x(\tau_D\le 1)\ge &\ \bP_x(X_1\in D^c)=\int_{D^c} p(1, x, z)dz\ge\int_{\{|z|\ge d_0\}} p(1, 0, z)dz\\
		\ge & \ \int_{\{|z^1|\ge d_0\}} p^\phi(1, 0, z^1) dz^1=2\int_{d_0}^\infty\frac{ds}{s\phi(s)} =:c_4(\phi, d_0)
	\end{align*}
	so that 
	$$\sup_{x\in D}\int_D p_D(1, x, y)dy=\sup_{x\in D}\bP_x(\tau_D>1)\le (1- c_4)=:c_5<1$$
	and  \eqref{eq:t-1} is bounded as follows: $t\in (n , n+1]\cap (2, \infty)$
	\begin{align*}
		\int_D p_D(t-1, x, z)dz\le c_5^n \int_{D} p_D(t-n-1, x_n, z) dz\le c_5^n\le c_5^{-1}\cdot c_5^t=c_5^{-1}\cdot e^{-t\ln c_5^{-1}}.
	\end{align*}
	Clearly, 
	$$\int_D p_D(t-1, x, z) dz\le 1\le e\cdot e^{-t}, \qquad \text{ for }t\in (1, 2).$$
	By letting $c_6:=\ln(c_5^{-1}\wedge e)$, we have that
	$$\int_D p_D(t-1, x, z) dz\le e\cdot e^{-t}\vee c_5^{-1}\cdot e^{-t\ln c_5^{-1}}\le (e\vee c_5^{-1}) e^{-c_6 t}, \quad \text{ for } (t, x)\in (1, \infty)\times D.$$
	Therefore, we obtain our result plugging the above in \eqref{eq:t}
	with $c_1:=c_3(e\vee c_5^{-1})$ and $c_2=c_4$.
	
\end{proof}

\medskip

Now we are ready to prove \autoref{t:nmain2}(2).

\medskip
{\it Proof of \autoref{t:nmain2}(2).}
Let $D$ be bounded and $\cL_D$ be the infinitesimal generator of the semigroup $\{P_t^D\}$ on $L^2(D)$. 
For each $t>0$, by \eqref{eq:dhke_u},  the heat kernel $p_D(t, x, y)$ is bounded in $D\times D$ and so $\cL_D$ is a Hilbert-Schmidt operator which is a compact operator. Therefore there exist the first eigenvalue $\lambda^D:=-\sup(\sigma(\cL_D))$ of multiplicity $1$ and the corresponding eigenfunction $f_D$ with unit $L^2$ norm, $\|f_D\|_{L^2(D)}=1$.
Since $\{P_t^D\}$ is the semigroup corresponding to $\cL_D$, 
\begin{align}\label{eq:ev}
	f_D(x)=e^{t \lambda^D }P_t^D f_D(x)=e^{t \lambda^D }\int_D p_D(t, x, y) f_D(y)dy.	
\end{align}
We obtain the upper bound of \eqref{eq:ev} when $t=\frac{1}{4}$ using \eqref{eq:hke}, \eqref{eq:dhke_u} with $T=1$ and the H\"older inequality that for all $x\in D$
\begin{align}\label{eq:l1}
	f_D(x)\le &\ c e^{\frac{1}{4}{\lambda^D }}\left(1\wedge 2\sqrt{\phi(\delta_D(x))}\right)\int_D p(\tfrac{1}{4}, x, y) f_D(y)dy\nn\\
	\le &\ 2c e^{\frac{1}{4}{\lambda^D }}\left(1\wedge \sqrt{\phi(\delta_D(x))}\right)\sqrt{\int_D p(\tfrac{1}{4}, x, y)^2 dy}\cdot \|f_D\|_{L^2(D)}\nn\\
	\le &\ 2ce^{\frac{1}{4}{\lambda^D }}\left(1\wedge \sqrt{\phi(\delta_D(x))}\right)\sqrt{p(\tfrac{1}{2}, x, x)}\nn\\
	\le &\ 2ce^{\frac{1}{4}{\lambda^D }}\left(1\wedge \sqrt{\phi(\delta_D(x))}\right)\left[\phi^{-1}(2^{-1})\right]^{-d/2}
	=: c^*e^{\frac{1}{4}{\lambda^D }}\left(1\wedge \sqrt{\phi(\delta_D(x))}\right).
\end{align}
for some $c^*=c^*(\kappa_0, \kappa_1, \eta,  \phi, R, \Lambda, \gamma, d)>0$.
On the other hand, if $t=1$ in \eqref{eq:ev}, \eqref{eq:hke} and \eqref{eq:dhke_l} with $T=1$ yields that
for all $x\in D$
\begin{align}\label{eq:l1_1}
	f_D(x)\ge &\ c e^{{\lambda^D }}\left(1\wedge \sqrt{\phi(\delta_D(x))}\right)\int_D\left(1\wedge \sqrt{\phi(\delta_D(y))}\right) p(1, x, y) f_D(y)dy\nn\\
	\ge &\ c e^{{\lambda^D }}\left(1\wedge \sqrt{\phi(\delta_D(x))}\right)\left(1\wedge \frac{1}{d_0^d\phi(d_0)^d}\right)\int_D \left(1\wedge \sqrt{\phi(\delta_D(y))}\right) f_D(y)dy\nn\\
	=: & \ c_*e^{{\lambda^D }}\left(1\wedge \sqrt{\phi(\delta_D(x))}\right)\int_D \left(1\wedge \sqrt{\phi(\delta_D(y))}\right) f_D(y)dy
\end{align}
for some $c_*=c_*(\kappa_0, \kappa_1, \eta,  \phi, R, \Lambda, \gamma, d_0, d)>0$  where $d_0:=\text{diam}(D)$.
Combining theses two inequalities,  since $\|f_D\|_{L^2(D)}=1$, we have that for any $x\in D$,
\begin{align}\label{eq:l2}
	f_D(x)\ge \frac{c_*}{c^*}e^{\tfrac{3}{4}{\lambda^D }}\left(1\wedge \sqrt{\phi(\delta_D(x))}\right)\int_D  f_D(y)^2dy= \frac{c_*}{c^*}e^{\tfrac{3}{4}{\lambda^D }}\left(1\wedge \sqrt{\phi(\delta_D(x))}\right).
\end{align}
By \eqref{eq:l1} and \eqref{eq:l2}, we have that 
\begin{align}\label{eq:lL}
	c_*(c^*)^{-1}e^{\tfrac{3}{4}{\lambda^D }}\le c^*e^{\frac{1}{4}{\lambda^D }}
	\quad \text{ which implies }\quad\lambda^D \le 2\ln((c^*)^2/c_*) <\infty,
\end{align}
and this prove the upper bound of \eqref{eq:La1}.
By \eqref{eq:l1}--\eqref{eq:l1_1} and \eqref{eq:lL}, for $x\in D$,
\begin{align}\label{eq:l3}
	c_0^{-1}\left(1\wedge \sqrt{\phi(\delta_D(x))}\right)\le f_D(x)\le c_0\left(1\wedge \sqrt{\phi(\delta_D(x))}\right)
\end{align}
for some $c_0=c_0(\kappa_0, \kappa_1, \eta,  \phi, R, \Lambda, \gamma, d_0, d)\ge 1$.
From \eqref{eq:ev}, 
$$1=\|f_D\|_{L^2(D)}^2=e^{t\lambda^D}\int_{D\times D} f_D(x) p_D(t, x, y)f_D(y)dxdy$$
and so that
$$\int_{D\times D} f_D(x) p_D(t, x, y)f_D(y)dxdy=e^{-t\lambda^D}.$$
Therefore \eqref{eq:l3} implies that for $t>0$,  
\begin{align}\label{eq:l6}
	c_0^{-2}e^{-t\lambda^D}
	\le \int_{D\times D} \left(1\wedge \sqrt{\phi(\delta_D(x))}\right) p_D(t, x, y)\left(1\wedge \sqrt{\phi(\delta_D(y))}\right)dxdy
	\le c_0^{2}e^{-t\lambda^D}.
\end{align}
For the lower bound of \eqref{eq:La1}, by \autoref{eq:Ondia}, there exists $c_i=c_i(\phi, d_0, d)>0$, $i=1,2$ such that 
$$p_D(t, x, y)\le c_1 e^{-c_2 t}, \qquad  \text{ for }(t, x, y)\in (0, \infty)\times D\times D.$$
Combining this with \eqref{eq:l6}, for all $t>0$,
\begin{align*}
	c_0^{-2}e^{-t\lambda^D}\le &\int_{D\times D} \left(1\wedge \sqrt{\phi(\delta_D(x))}\right) p_D(t, x, y)\left(1\wedge \sqrt{\phi(\delta_D(y))}\right)dxdy\\
	\le &\ c_1e^{-c_2t}\int_{D\times D} \left(1\wedge \sqrt{\phi(\delta_D(x))}\right) \left(1\wedge \sqrt{\phi(\delta_D(y))}\right)dxdy	\le  c_1e^{-c_2t}|D|^2.
\end{align*}
Therefore, we have 
$$e^{(c_2-\lambda^D )t}\le (c_0)^2c_1|D|^2<\infty\quad\text{ which implies } \quad\lambda^D\ge c_2.$$
Now we will prove \eqref{eq:La2}.
For any $T>0$, let $t_0:=T/4$. 
By \eqref{eq:dhke_u} and \eqref{eq:dhke_l}, there is a constant $c_3=c_3(\kappa_0, \kappa_1, \eta,  \phi, R, T, \Lambda, \gamma,  d)\ge 1$ such that for $(u, v)\in D\times D$, 
\begin{align}
	p_D(t_0, u, v)\le&  \ c_3  [\phi^{-1}(t_0)]^{-d}t_0^{-1}\left(1\wedge \sqrt{\phi(\delta_D(u))}\right)\left(1\wedge \sqrt{\phi(\delta_D(v))}\right)\,,\label{eq:l4}\\
	p_D(t_0, u, v)\ge& \ c_3^{-1}  \left([\phi^{-1}(t_0)]\wedge \frac{t_0}{d_0\phi(d_0)}\right)^d \left(1\wedge \sqrt{\phi(\delta_D(u))}\right)\left(1\wedge \sqrt{\phi(\delta_D(v))}\right).\label{eq:l5}
\end{align}
By the semigroup property with \eqref{eq:l6} and \eqref{eq:l4},
for all $(t, x, y)\in (T, \infty)\times D\times D$,
\begin{align}\label{eq:l7}
	p_D(t, x, y)=& \int_{D\times D} p_D(t_0, x, u)p_D(t-2t_0, u, v)p_D(t_0, v, y) dudv\nn\\
	\le & \ c_3^2 [\phi^{-1}(t_0)]^{-2d}t_0^{-2} \sqrt{\phi(\delta_D(x))}\sqrt{\phi(\delta_D(y))}\nn\\
	&\qquad\cdot\int_{D\times D} \left(1\wedge \sqrt{\phi(\delta_D(u))}\right) p_D(t-2t_0, u, v) \left(1\wedge \sqrt{\phi(\delta_D(v))}\right)dudv\nn\\
	\le &  (c_0c_3)^2 e^{-(t-2t_0)\lambda^D}[\phi^{-1}(t_0)]^{-2d}t_0^{-2} \sqrt{\phi(\delta_D(x))}\sqrt{\phi(\delta_D(y))}
\end{align}
Similarly,  using \eqref{eq:l6} and \eqref{eq:l5},
for all $(t, x, y)\in (T, \infty)\times D\times D$, since $D$ is bounded,
\begin{align}\label{eq:l8}
	p_D(t, x, y)
	\ge & \ c_3^{-2}  \left([\phi^{-1}(t_0)]\wedge \frac{t_0}{d_0\phi(d_0)}\right)^{2d} \sqrt{\phi(\delta_D(x))}\sqrt{\phi(\delta_D(y))}\nn\\
	&\qquad\cdot \int_{D\times D} \left(1\wedge \sqrt{\phi(\delta_D(u))}\right) p_D(t-2t_0, u, v) \left(1\wedge \sqrt{\phi(\delta_D(v))}\right)dudv\nn\\
	\ge & \ (c_0c_3)^2 e^{-(t-2t_0)\lambda^D}\left([\phi^{-1}(t_0)]\wedge \frac{t_0}{d_0\phi(d_0)}\right)^{2d}\sqrt{\phi(\delta_D(x))}\sqrt{\phi(\delta_D(y))}.
\end{align}
Therefore, we conclude \eqref{eq:La2} by \eqref{eq:l7}  with \eqref{eq:lL} and \eqref{eq:l8}.
\hfill 
\rule{0.5em}{0.5em}

\section{Green function estimates}\label{sec:green}

We first obtain the weak scaling condition for the inverse function $\phi^{-1}$ by {\bf (WS)} that 
\begin{align}\label{e:invP}
	\left(\frac{t}{s\uC}\right)^{1/\ua}\le \frac{\phi^{-1}(t)}{\phi^{-1}(s)}\le \left(\frac{t}{ s\lC}\right)^{1/\la} \qquad\text{ for $0<s\le t$.}
\end{align}

\medskip

{\it Proof of \autoref{T:green_new}. } 
When $d=1$, we obtain the result by  \cite[Thoerem 1.6]{GKK20}. So we only consider when $d\ge 2$.
Suppose that $x, y\in D$ satisfying that $\min_{i\in \{1, \ldots, d\}}|x^i-y^i|>0$.
Let $r_1:=\min_{i\in \{1, \ldots, d\}}|x^i-y^i|, r_2:=\max_{i\in \{1, \ldots, d\}}|x^i-y^i|$  and 
$T:=\phi(\text{diam}(D))$.
We may assume that $r_1<r_2$, and the Green function is decomposed according to the range of $t>0$  that  
\begin{align*}
	G_D(x, y)&=\int_0^\infty p_D(t, x, y) d t\nn\\
	&=\int_0^{\phi(r_1)}p_D(t, x, y) d t+\int_{\phi(r_1)}^{\phi(r_2)}p_D(t, x, y) d t+\int_{\phi( {r_2})}^{T}p_D(t, x, y) d t+\int_T^\infty p_D(t, x, y) d t\\
	&=:I_0+I_1+ I_2+I_3.
\end{align*}
We first note that \eqref{eq:La2} in \autoref{t:nmain2}(2) implies
\begin{align*}
	I_3=\int_T^\infty p_D(t, x, y) dt \asymp \sqrt{\phi(\delta_D(x))}\sqrt{\phi(\delta_D(y))}.
\end{align*}
We need to prove that there exists $c_0>0$ such that
\begin{align*}
	I_{0}&\asymp
	\left(1\wedge\frac{\sqrt{\phi(\delta_D(x))}}{\sqrt{\phi(r_1)}}\right)\left(1\wedge\frac{\sqrt{\phi(\delta_D(y))}}{\sqrt{\phi(r_1)}}\right) \phi(r_1)^{d+1}\prod_{i=1}^d\frac{1}{|x^i-y^i|\phi(|x^i-y^i|)}\,,\\
	I_1&\le c_0\left(1\wedge\frac{\sqrt{\phi(\delta_D(x))}}{\sqrt{\phi(r_2)}}\right)\left(1\wedge\frac{\sqrt{\phi(\delta_D(y))}}{\sqrt{\phi(r_2)}}\right) \phi(r_2)^{d+1}\prod_{i=1}^d\frac{1}{|x^i-y^i|\phi(|x^i-y^i|)}
	\,, \\
	I_2&  \le c_0 \left(1\wedge\frac{\sqrt{\phi(\delta_D(x))}}{\sqrt{\phi(r_2)}}\right)\left(1\wedge\frac{\sqrt{\phi(\delta_D(y))}}{\sqrt{\phi(r_2)}}\right) \phi(r_2)^{d+1}\prod_{i=1}^d\frac{1}{|x^i-y^i|\phi(|x^i-y^i|)}.
\end{align*}
\medskip

{\bf  Asymptotic estimates   of $I_0$:}
For $0< t<\phi(r_1)\le \phi(|x^i-y^i|)<T, \, i=1,\dots,d$, by letting $v_1:=\frac{\phi(r_1)}{t}\ge 1$, we have that 
\begin{align*}
	\prod_{i=1}^d\left(\phi^{-1}(t)\wedge \frac{t }{|x^i-y^i|\phi(|x^i-y^i|)}\right)d t
	&\asymp \prod_{i=1}^d \frac{t}{|x^i-y^i|\phi(|x^i-y^i|)}d t\\
	&=\prod_{i=1}^d \frac{\phi(r_1)}{v_1|x^i-y^i|\phi(|x^i-y^i|)}\left(-\frac{\phi(r_1)}{v_1^2}\right)d v_1.
\end{align*} 
By change of variables with $v_1=\frac{\phi(r_1)}{t}\ge 1$,
\autoref{t:nmain}, \autoref{t:nmain2}(1) and \eqref{e:invP} imply that there exists $c_1>1$ such that 

\begin{align*}
	&c_1^{-1} \phi(r_1)^{d+1}\prod_{i=1}^d\frac{1}{|x^i-y^i|\phi(|x^i-y^i|)}\left(1\wedge\frac{\sqrt{\phi(\delta_D(x))}}{\sqrt{\phi(r_1)}}\right)\left(1\wedge\frac{\sqrt{\phi(\delta_D(y))}}{\sqrt{\phi(r_1)}}\right)\nn\\
	\le&\ \phi(r_1)^{d+1}\prod_{i=1}^d\frac{1}{|x^i-y^i|\phi(|x^i-y^i|)}\left(1\wedge\frac{\sqrt{\phi(\delta_D(x))}}{\sqrt{\phi(r_1)}}\right)\left(1\wedge\frac{\sqrt{\phi(\delta_D(y))}}{\sqrt{\phi(r_1)}}\right)
	\int_{1}^{\infty} 	v_1^{-d-2}d v_1
	\nn\\
	\le&\ 	I_0 \asymp \int_{1}^{\infty}\left(1\wedge\frac{\sqrt{v_1\phi(\delta_D(x))}}{\sqrt{\phi(r_1)}}\right)\left(1\wedge\frac{\sqrt{v_1\phi(\delta_D(y))}}{\sqrt{\phi(r_1)}}\right)\prod_{i=1}^d \frac{\phi(r_1)}{v_1|x^i-y^i|\phi(|x^i-y^i|)}\left(\frac{\phi(r_1)}{v_1^2}\right)d v_1\nn\\
	\le&\
	\phi(r_1)^{d+1}\prod_{i=1}^d\frac{1}{|x^i-y^i|\phi(|x^i-y^i|)}\left(1\wedge\frac{\sqrt{\phi(\delta_D(x))}}{\sqrt{\phi(r_1)}}\right)\left(1\wedge\frac{\sqrt{\phi(\delta_D(y))}}{\sqrt{\phi(r_1)}}\right)
	\int_{1}^{\infty} 	v_1^{-d-1}d v_1
	\nn\\
	\le&\
	c_1\phi(r_2)^{d+1}\prod_{i=1}^d\frac{1}{|x^i-y^i|\phi(|x^i-y^i|)}\left(1\wedge\frac{\sqrt{\phi(\delta_D(x))}}{\sqrt{\phi(r_2)}}\right)\left(1\wedge\frac{\sqrt{\phi(\delta_D(y))}}{\sqrt{\phi(r_2)}}\right).
\end{align*}	
\medskip

{\bf Upper bound estimate of $I_1$: }
For $\phi(r_1)\le t<\phi(r_2)$, let $1<\frac{\phi(r_2)}{t}=:v_2$.
By the change of variable with $v_2=\frac{\phi(r_2)}{t}$, we have that 
\begin{align*}
	\prod_{i=1}^d\left([\phi^{-1}(t)]^{-1}\wedge \frac{t}{|x^i-y^i|\phi(|x^i-y^i|)}\right)d t
	&\le \prod_{i=1}^d \frac{t}{|x^i-y^i|\phi(|x^i-y^i|)} d t\nn\\
	&\le\prod_{i=1}^d \frac{\phi(r_2)}{v_2|x^i-y^i|\phi(|x^i-y^i|)}\left(-\frac{\phi(r_2)}{v_2^2}\right)d v_2.
\end{align*} 
Then since $v_2>1$, \autoref{t:nmain} implies that
\begin{align*}
	I_1
	&\le  \int^{1}_{\frac{\phi(r_2)}{\phi(r_1)}}\left(1\wedge\frac{\sqrt{v_2\phi(\delta_D(x))}}{\sqrt{\phi(r_2)}}\right)\left(1\wedge\frac{\sqrt{v_2\phi(\delta_D(y))}}{\sqrt{\phi(r_2)}}\right)\prod_{i=1}^d \frac{\phi(r_2)}{v_2|x^i-y^i|\phi(|x^i-y^i|)}\left(-\frac{\phi(r_2)}{v_2^2}\right)d v_2\\
	&\le
	{\phi(r_2)^{d+1}}
	\prod_{i=1}^d\frac{1}{|x^i-y^i|\phi(|x^i-y^i|)} \left(1\wedge\frac{\sqrt{\phi(\delta_D(x))}}{\sqrt{\phi(r_2)}}\right)\left(1\wedge\frac{\sqrt{\phi(\delta_D(y))}}{\sqrt{\phi(r_2)}}\right)\hspace{-.05in}
	\int_{1}^{\frac{\phi(r_2)}{\phi(r_1)}}\hspace{-.05in} v_2^{-d-1}d v_2\\
	&\le  c_2
	\left(1\wedge\frac{\sqrt{\phi(\delta_D(x))}}{\sqrt{\phi(r_2)}}\right)\left(1\wedge\frac{\sqrt{\phi(\delta_D(x))}}{\sqrt{\phi(r_2)}}\right)
	\phi(r_2)^{d+1}
	\prod_{i=1}^d\frac{1}{|x^i-y^i|\phi(|x^i-y^i|)}.
\end{align*}	
\medskip

{\bf Upper bound estimate of $I_2$:}
For $\phi(r_2)\le t<T$,  
\begin{align*}
	\prod_{i=1}^d\left([\phi^{-1}(t)]^{-1}\wedge \frac{t}{|x^i-y^i|\phi(|x^i-y^i|)}\right) dt
	\asymp  [\phi^{-1}(t)]^{-d} dt= \left(\frac{r_2}{r_2\phi^{-1}(t)}\right)^{d} dt.
\end{align*} 
Let $u_2:=\frac{\phi(r_2)}{t}\le  1$.
Then $r_2=\phi^{-1}(u_2t)$.
By the change of variable with $u_2=\frac{\phi(r_2)}{t}$, \eqref{e:invP} and the fact that $r_{i}\le r_{i+1}$  and $\phi(r_{i})\le \phi(r_{i+1})$,  \autoref{t:nmain} implies that
\begin{align*}
	I_k& 	\asymp\int^{\phi(r_2)/T}_1 \left(\frac{\phi^{-1}(u_2t)}{r_2\phi^{-1}(t)}\right)^{d} \left(1\wedge\frac{\sqrt{u_2\phi(\delta_D(x))}}{\sqrt{\phi(r_2)}}\right)\left(1\wedge\frac{\sqrt{u_2\phi(\delta_D(y))}}{\sqrt{\phi(r_2)}}\right)	\left(-\frac{\phi(r_2)}{u_2^2}\right)d u_2\nn\\
	& \le c_3 \frac{\phi(r_2)}{(r_2)^d} \left(1\wedge\frac{\sqrt{\phi(\delta_D(x))}}{\sqrt{\phi(r_2)}}\right)\left(1\wedge\frac{\sqrt{\phi(\delta_D(y))}}{\sqrt{\phi(r_2)}}\right)\int^{1}_{0}
	u_2^{ d/\ua-2}d u_2\nn\\
	& \le c_4 \frac{\phi(r_2)}{(r_2)^d} \left(1\wedge\frac{\sqrt{\phi(\delta_D(x))}}{\sqrt{\phi(r_2)}}\right)\left(1\wedge\frac{\sqrt{\phi(\delta_D(y))}}{\sqrt{\phi(r_2)}}\right)\nn\\
	&\le c_4 \left(1\wedge\frac{\sqrt{\phi(\delta_D(x))}}{\sqrt{\phi(r_2)}}\right)\left(1\wedge\frac{\sqrt{\phi(\delta_D(y))}}{\sqrt{\phi(r_2)}}\right)
	{\phi(r_2)^{d+1}}\prod_{i=1}^d\frac{1}{|x^i-y^i|\phi(|x^i-y^i|)}.
\end{align*}\hfill 
\rule{0.5em}{0.5em}

\medskip

\begin{rem}\label{R:Green_new}
	Suppose that $x^i\ne y^i$ for all $i\in \{1, \ldots, d\}$ and $|x^{i_1}-y^{i_1}|\neq  |x^{i_2}-y^{i_2}|$ for $i_1, i_2\in E\subset \{1, 2,\ldots, d\}$.
	Let $\sigma:\{1, \ldots, d\}\to \{1, \ldots, d\}$ be a permutation to obtain 
	$0<|x^{\sigma(i)}-y^{\sigma(i)}|<|x^{\sigma(j)}-y^{\sigma(j)}|$ and $r_{\sigma(i)}:=|x^{\sigma(i)}-y^{\sigma(i)}|, i\in E$. 
	Using the similar proof as in \autoref{T:green_new}, i.e., decomposing the time variable as $0<\phi(r_{\sigma(i)})<\phi(r_{\sigma(j)})<T<\infty$,
	when $d\ge 2$, we can obtain a more accurate upper bound that
	\begin{align*}
		&\	G_D(x, y)\\
		\le	&\  c \sum_{i\in E}\left(
		\Big(1\wedge\frac{\sqrt{\phi(\delta_D(x))}}{\sqrt{\phi(|x^i-y^i|)}}\Big)\Big(1\wedge\frac{\sqrt{\phi(\delta_D(y))}}{\sqrt{\phi(|x^i-y^i|)}}\Big)	\phi(|x^i-y^i|)^{d+1}\right)\prod_{i=1}^d\frac{1}{|x^{i}-y^{i}|\phi(|x^{i}-y^{i}|)}\\
		\asymp & \ \sum_{i=1}^d\left(
		\Big(1\wedge\frac{\sqrt{\phi(\delta_D(x))}}{\sqrt{\phi(|x^i-y^i|)}}\Big)\Big(1\wedge\frac{\sqrt{\phi(\delta_D(y))}}{\sqrt{\phi(|x^i-y^i|)}}\Big)	\phi(|x^i-y^i|)^{d+1}\right)\prod_{i=1}^d\frac{1}{|x^{i}-y^{i}|\phi(|x^{i}-y^{i}|)}.
	\end{align*}
	
\end{rem}

\bigskip

\textbf{Acknowledgment.} The authors are grateful for helpful discussions with Prof. Eryan Hu during the First China-Japan-Korea Probability Workshop in May 2019.


\bibliographystyle{alpha}
\bibliography{DHKE-anisotropic}

\makeatletter
\providecommand
\@dotsep{5}
\makeatother
\relax
	
\end{document}